\documentclass[reqno]{amsart}

\usepackage[T1]{fontenc}
\usepackage{amsmath,amsfonts,amssymb}
\usepackage{xcolor}
\usepackage{cite}
\usepackage{microtype}

\usepackage[hidelinks]{hyperref}

\allowdisplaybreaks

\newtheorem{theorem}{Theorem}[section]
\newtheorem{lemma}[theorem]{Lemma}
\newtheorem{corollary}[theorem]{Corollary}
\newtheorem{definition}[theorem]{Definition}

\theoremstyle{definition}
\newtheorem{remark}[theorem]{Remark}

\numberwithin{equation}{section}

\newcommand{ \mr }{ \mathbb{R} }

\def\Xint#1{\mathchoice
	{\XXint\displaystyle\textstyle{#1}}%
	{\XXint\textstyle\scriptstyle{#1}}%
	{\XXint\scriptstyle\scriptscriptstyle{#1}}%
	{\XXint\scriptstyle\scriptscriptstyle{#1}}%
	\!\int}
\def\XXint#1#2#3{{\setbox0=\hbox{$#1{#2#3}{\int}$}
		\vcenter{\hbox{$#2#3$}}\kern-.5\wd0}}

\begin{document}

	\title[Very Weak Solutions to Elliptic Equations with Matrix Weights]{Gradient Estimates Near the Natural Exponent for Very Weak Solutions to
		$A_p$-Weighted Quasilinear Elliptic Equations}
	
	\author[S.-S. Byun, M. Lim]{Sun-Sig Byun\and Minkyu Lim}

	\address{Sun-Sig Byun: Seoul National University, Research Institute of Basic Sciences, Seoul 08826, Korea}
	\email{byun@snu.ac.kr}

	\address{Minkyu Lim: Seoul National University, Research Institute of Basic Sciences, Seoul 08826, Korea}
	\email{mk0314@snu.ac.kr}

	\thanks{S. Byun was supported by the National Research Foundation of Korea (NRF) grant funded by	the Korean government(2022R1A2C1009312).
	 M. Lim was supported by the Institute of Basic Science, Seoul National University(RS-2019-NR040082) and by Basic Science Research Program(RS-2023-00247402) through the 
		National Research Foundation of Korea(NRF) funded by the Ministry of Education.}


	
	\begin{abstract}
		
		We establish local Calder\'on--Zygmund estimates near the natural exponent for very weak solutions to matrix-weighted quasilinear equations
		\[
		-	\mathrm{div\,} A_{\mathbb{M}}(x,Du)
		=
		-	\mathrm{div\,} A_{\mathbb{M}}(x,\mathbf{f}),
		\qquad
		A_{\mathbb{M}}(x,\xi)=\mathbb{M}(x)A(x,\mathbb{M}(x)\xi),
		\]
		where $A$ has $p$-growth and strong monotonicity, and $\mathbb{M}$ is a
		measurable positive-definite matrix field. We assume that $\mathbb{M}$ has bounded
		condition number and that \(\omega:=|\mathbb{M}|^p\in A_p\), without imposing uniform
		upper or lower bounds on $\mathbb{M}$.
	    This extends the near-natural Calderón--Zygmund theory developed by Adimurthi--Phuc \cite{AP15} to the matrix-degenerate setting: There exists $\delta_0>0$ such that every very weak solution $u \in W^{p-\delta_{0}}_{\omega, \mathrm{loc}}$ satisfies
		\[
		\mathbf{f}\in L^\gamma_{\omega,\mathrm{loc}} 
		\Longrightarrow
		Du\in L^\gamma_{\omega,\mathrm{loc}}
		\]
		for $p-\delta_0\le\gamma\le p+\delta_0$. The proof requires handling the lack of energy estimates below the natural exponent and the use of Lipschitz truncation in the weighted setting. We achieve this through comparison estimates, higher integrability, and weighted analysis techniques.
	\end{abstract}

	\maketitle

	\section{Introduction}\label{Sec1}  
	We consider the following prototype nonlinear elliptic equation with spatially varying degenerate coefficients:
	\begin{equation}\label{modeleq}
		-	\mathrm{div\,}(|\mathbb{M}(x)Du|^{p-2}
		\mathbb{M}^{2}(x)Du) =-	\mathrm{div\,}(|\mathbb{M}(x)\mathbf{f}|^{p-2}
		\mathbb{M}^{2}(x)\mathbf{f}) \quad \textrm{in}\ \Omega,
	\end{equation}
	where $\Omega \subset \mr^n$ is an open domain and $\mathbb{M}: \Omega \rightarrow \mr^{n\times n}$ is a symmetric and positive definite matrix-valued weight. These types of equations have received a lot of attention in recent years due to the  analytical challenges they present \cite{BCR26, BCL24,BBDL23,BDGP22,P20,CMP18,BL25, DFRZ25, DFMRZ23}. If $\mathbb{M}(x)$ is of the form $ \mathbb{M}(x) =  \omega(x)^{\frac{1}{p}} I $, then the equation reduces to the following model equation:
	\begin{equation*}
		-	\mathrm{div\,}(\omega(x)|Du|^{p-2}Du) =-	\mathrm{div\,}(\omega(x)|\mathbf{f}|^{p-2}\mathbf{f}) \quad \textrm{in}\ \Omega.
	\end{equation*}
	
	The main difficulty lies in the fact that uniform ellipticity may break down in two distinct and interacting ways. The \(p\)-Laplace structure exhibits a strong gradient-dependent nonlinearity, causing degeneracy or singularity as $|D u|\to 0$, while a possibly vanishing coefficient $\omega(x)$ further weakens the ellipticity independently. That is, the loss of ellipticity is not confined to a fixed region but instead depends on both the spatial inhomogeneity and the behavior of the solution itself. As a result, this nontrivial interaction necessitates the use of weighted Sobolev spaces and complicates the corresponding regularity theory. Moreover, anisotropy arising from a general matrix, including rotations of the principal directions, introduces an additional difficulty.
	
    The study of elliptic equations with degenerate weights originates from the classical work of Fabes, Kenig, and Serapioni \cite{FKS82}, where Hölder regularity of weak solutions was first established. This work initiated the modern study of degenerate elliptic equations in weighted settings and introduced weighted Poincaré inequalities in this context. These continuity results have been further developed in various directions in \cite{CDS11, CMN13, IM19, DFRZ25}.

	Calderón–Zygmund type estimates are a classical topic in the regularity theory of elliptic partial differential equations, which remains an active area of research in modern PDE. They have been extensively studied in both classical uniformly elliptic equations and more general settings involving degeneracies, nonstandard growth conditions, and nonlocal equations. For elliptic equations with degenerate weights like \eqref{modeleq}, Cao, Mengesha, and Phan \cite{CMP18} established global weighted Calderón–Zygmund estimates 
	\begin{align}\label{impli}
	|\mathbf{f}|^{\gamma} \omega \in L^{1} \implies |Du|^{\gamma} \omega \in L^{1}
	\end{align}
	for $\gamma>p=2$ under Muckenhoupt weight assumptions $\omega \in A_{2}$ assuming in addition a small BMO-type regularity on the underlying structure.  Meanwhile, Balci–Diening–Giova–Passarelli di Napoli \cite{BDGP22} adopted a different perspective by representing the weight $\omega$ through a multiplier $\mathbb{M}$ rather than treating it as a measure, and established that
	$$
	|\mathbb{M}\mathbf{f}|   \in L^{\gamma}_{\mathrm{loc}} \implies |\mathbb{M}Du|   \in L^{\gamma}_{\mathrm{loc}}
	$$
	for $\gamma>p$ in the context of elliptic equations with matrix-valued degenerate weights, also under analogous small BMO assumptions on the coefficients. These results, however, do not
	address the very weak regime in which the gradient is initially
	integrable only below the natural exponent $p$.
	
	In this paper, we aim to study local gradient estimates considered above for exponents slightly below and above the natural integrability exponent $p$. More precisely, under suitable structural assumptions on the operator and the underlying weight, we investigate whether the  implication \eqref{impli} can be extended to the range
	$$
	p-\delta <\gamma <p+\delta,
	$$
	for some small $\delta>0$. Gradient estimates for very weak solutions have been studied in a variety of elliptic settings; see, e.g., \cite{IS94,KL02 , P11, CT25, L93, BDS16, BS16}. To the best of our knowledge, this is the first result to establish gradient estimates for very weak solutions under such a low integrability assumption on the divergence data in the setting of matrix-weighted degenerate elliptic equations. This naturally leads us to examine which of the two formulations is better suited to extending the estimates below the natural exponent $p$.
	
	The main point of our approach is that the estimate is obtained
	without imposing a small-log-BMO assumption on the matrix multiplier.
	Instead, the degeneracy is encoded in the $A_p$ weight
	$\omega=|\mathbb M|^p$, and the proof requires a careful interaction
	between weighted higher integrability, comparison estimates, and
	Lipschitz truncation. This is particularly relevant in the very weak
	regime, where the initial integrability of $Du$ lies below the natural
	exponent $p$.
	
	For this reason, we formulate the problem in the measure-theoretic
	setting. This formulation is particularly well suited to very weak
	solutions, since the local integrability of
	$|Du|^{p-1}\omega$, as required in their definition, is naturally
	expressed in terms of the weighted measure. By contrast, in the
	multiplier formulation, $\mathbb M$ and $Du$ are intrinsically coupled
	and must be controlled simultaneously, which makes the direct use of
	Lipschitz truncation less natural. We note, however, that a
	multiplier-based approach is also possible, although technically more
	involved; see Remark \ref{multip}.
	
	A main feature of our approach is that no small-log-BMO assumption is needed for the matrix multiplier, in contrast to some Calderón–Zygmund estimates where such a condition is imposed to control the coefficient regularity. Our goal is instead to obtain integrability estimates in a low-regularity setting, where the $A_p$ structure of the weight provides the relevant information without requiring additional regularity of the multiplier. The analysis makes extensive use of weighted analysis, including the structural properties of Muckenhoupt weights, weighted Poincaré inequalities, weighted Gehring lemma.

	The rest of the paper is organized as follows. In Section \ref{Sec2}, we introduce the notation and state the main results, including the formulation of the underlying PDE, the relevant Muckenhoupt weight classes, weighted Sobolev spaces, and the notion of very weak solutions, together with our main theorem. Section \ref{Sec3} is devoted to preliminary results, including a weighted Sobolev–Poincaré inequality and the Lipschitz truncation lemma. In Section \ref{Sec4}, we establish higher integrability of very weak solutions, derive a priori estimates, and prove the existence of solutions to an associated homogeneous problem. Section \ref{Sec5} is devoted to the proof of the main estimates in the measure settings.

	\section{Notation and Main Results}\label{Sec2} 
	In this section, we introduce the notation and state the main results precisely. 
	
	We define standard notations that will be used frequently throughout the paper. Let $B_\rho(y)$ denote the open ball in $\mathbb{R}^n$ with center $y \in \mathbb{R}^n$ and radius $\rho>0$. When the center $y$ is clear from the context, we simply write $B_\rho$ instead of $B_\rho(y)$. For an integrable function $v$ defined on a bounded measurable set $E \subset \mathbb{R}^n$, we denote by
	\[
	v_E := \Xint-_E v(x)\,dx = \frac{1}{|E|}\int_E v(x)\,dx
	\]
	the average of $v$ over $E$, where $|E|$ denotes the Lebesgue measure of $E$.
	If $v$ is defined only on a subset $F \subset E$, we extend $v$ by zero outside $F$ when necessary. We denote by  $\chi_{E}$
	the characteristic function of a set $E$.

	As a generalization of the model equation in \eqref{modeleq}, we consider the following structure.
	Let $1 < p < \infty$ and $\Omega \subset \mr^n$ be an open domain with $n \geq 2$, and let $ A(x, \xi) : \Omega \times \mathbb{R}^n \to \mathbb{R}^n $	be a Carathéodory map. Assume that there exist constants 
	\[
	0 < \nu \leq L < \infty
	\]
	such that for all $x \in \Omega$ and $\xi, \zeta \in \mathbb{R}^n$ the following structure conditions hold:
	\begin{align}\label{struc}
		\notag A(x,\xi) \cdot \xi &\ge \nu |\xi|^p, \\
		\notag |A(x,\xi)| &\le L |\xi|^{p-1}, \\
		\bigl(A(x,\xi)-A(x,\zeta)\bigr)\cdot (\xi-\zeta) 
		&\ge \nu (|\xi|+|\zeta|)^{p-2}|\xi-\zeta|^2.
	\end{align}
	
	We will frequently work with the vector field defined by
	\begin{equation*}
		V(z) := V_{p}(z) =  |z|^{\frac{p-2}{2}} z.
	\end{equation*}
	Using the notation $V(z)$, equation \eqref{struc} can be rewritten as 
	$$
	\bigl(A(x,\xi)-A(x,\zeta)\bigr)\cdot (\xi-\zeta)  \geq  c |V(\xi) - V(\zeta) |^{2}.
	$$
	for some positive constant $c$ depending on $\nu,n,$ and $p$. 
	
	Let $\mathbb{M}: \mr^n \rightarrow \mr^{n\times n}$ be a symmetric and positive definite
	matrix-valued weight and we define the matrix-weighted operator
	\[
	A_{\mathbb{M}}(x,\xi) := \mathbb{M}(x)\, A\bigl(x, \mathbb{M}(x)\xi \bigr).
	\]
	 Here, $\mathbb{M}$ represents the degeneracy of the problem. Although one may consider the matrix $\mathbb{M}$ defined on a domain $\Omega$, we assume that the matrix is defined on the whole space $\mathbb{R}^n$. This assumption is made to avoid technical issues concerning the extension of matrix from $\Omega$ to $\mathbb{R}^n$, which are not the focus of this work. To ensure uniformly bounded anisotropy, we assume that the condition number of
	 $\mathbb{M}$ is uniformly bounded, namely,
	 \begin{equation}\label{unidegen}
	 	|\mathbb{M}(x)|  |\mathbb{M}^{-1}(x)| \leq \Lambda,
	 \end{equation}
	 where $|\cdot|$ denotes the  operator norm on the space of matrices.
	 Then, it follows that for any $\xi \in \mathbb{R}^n$, 
	 \begin{equation*}
	 \Lambda^{-1}\,|\mathbb{M}(x)| \,|\xi|
	 \;\le\;
	 |\mathbb{M}(x)\xi|
	 \;\le\;
	 |\mathbb{M}(x)| \,|\xi|.
	 \end{equation*}
	 This property allows us to describe the degeneracy in terms of a scalar weight. For convenience, we define $\omega(x) := |\mathbb{M}(x)|^p$.
	 
	  We now consider the following degenerate nonlinear elliptic problem, which is the main equation of this paper.
	 \begin{equation}\label{maineq}
	 	-\operatorname{div} A_{\mathbb{M}}(x,D u)
	 	=
	 	-\operatorname{div} A_{\mathbb{M}}(x,\mathbf{f})
	 	\qquad \text{in } \Omega,
	 \end{equation}
	 where $\mathbf{f}: \Omega \rightarrow \mr^n $ is a vector-valued function.
	  The operator  $A_{\mathbb{M}}$ satisfies the following weighted structure conditions in terms of the intrinsic variable $\mathbb{M}\xi$:
	\begin{align*}
		A_{\mathbb{M}}(x,\xi)\cdot \xi 
		= A(x, \mathbb{M}\xi)\cdot \mathbb{M} \xi 
		& \ge \nu |\mathbb{M} \xi|^p \geq \nu \Lambda^{-p} \omega(x) | \xi|^p, \\[0.3em]
		|A_\mathbb{M}(x,\xi)| 
		\le |\mathbb{M}(x)|\, |A(x,\mathbb{M}\xi)| 
		& \le L \omega(x) | \xi|^{p-1}, \\[0.3em]
		\bigl(A_{\mathbb{M}}(x,\xi)-A_{\mathbb{M}}(x,\zeta)\bigr)\cdot (\xi-\zeta)
		&\ge \nu (|\mathbb{M}\xi|+|\mathbb{M}\zeta|)^{p-2}
		|\mathbb{M}(\xi-\zeta)|^2   \\
		&\ge c (p, \nu, \Lambda) \omega |V(\xi) - V(\zeta) |^{2}. 
	\end{align*}
	Hence, the operator $A_{\mathbb{M}}$ can be viewed as $p$-monotone in the intrinsic variable
	$$
	z := \mathbb{M}\xi.
	$$ 
	Although the operator $A_{\mathbb{M}}$ satisfies scalar weighted structure conditions with respect to $\omega$, the present formulation should not be viewed merely as a special case of a general scalar weighted framework. Indeed, if a given vector field $B(x,\xi)$ satisfies the above weighted structure conditions with respect to $\omega$, it can be represented in the present matrix framework $B = A_{\mathbb{M}}$ by defining
$$
A(x, \xi) := \omega(x)^{-\frac{1}{p}} I \,  B( x, \omega(x)^{-\frac{1}{p}} I   \, \xi)
$$
with the choice $\mathbb{M} = \omega(x)^{\frac{1}{p}} I  $.
	 Under the bounded condition-number assumption, the quantitative degeneracy of the matrix coefficient is equivalent to that of an appropriate scalar weight. This allows us to work within a scalar weighted framework while keeping the result applicable to the matrix-multiplier model. We adopt the matrix formulation, as it provides a natural framework for comparing the measure-theoretic and multiplier perspectives, and is also well suited to variational formulations.

	We briefly recall the definition and some basic properties of Muckenhoupt weights, which play a central role in the study of weighted Sobolev spaces. Let $\omega : \mr^n \to [0,\infty)$ be a locally integrable function, referred to as a weight. A weight $\omega$ is called a Muckenhoupt $A_s$ weight if
	$$
	[\omega]_{A_s} := \sup_{B_{\rho} \subset \mathbb{R}^n}
	\left(  \Xint-_{B_{\rho}} \omega(x)\,dx \right)
	\left(  \Xint-_{B_{\rho}} \omega(x)^{-\frac{1}{s-1}}\,dx \right)^{s-1}
	< \infty,
	$$
	where the supremum is taken over all balls $B \subset \mathbb{R}^n$. The quantity $[\omega]_{A_s}$ is called the $A_s$-constant of $\omega$.
	One of the key features of $A_p$ weights is the boundedness of the Hardy--Littlewood maximal operator. Namely, if $w \in A_p$, then there exists a constant $C > 0$ depending only on $n, p$ and  $	[\omega]_{A_p}$ such that
	\[
	\|\mathcal{M}(f)\|_{L_{\omega}^{p}} \leq C \|f\|_{L_{\omega}^{p}},
	\]
	for all $f \in L_{\omega}^p(\mathbb{R}^n)$, where $\mathcal{M}$ denotes the Hardy--Littlewood maximal function. This property is fundamental in establishing various weighted norm inequalities. Whenever a function is defined only on a domain $\Omega \subset \mathbb R^n$, we identify it with its zero extension outside $\Omega$. Accordingly, $\mathcal M(f)$ always denotes the Hardy--Littlewood maximal function of this extension.
	
	We now introduce the function space in which the solution to \eqref{maineq} naturally lies.	We begin by recalling the definition of the weighted Lebesgue space. For a weight $\omega$ and $1 \le p < \infty$, we define
	\[
	L^p_{\omega}(\Omega) := \left\{ u : \Omega \to \mathbb{R} \ \text{measurable} \;:\; 
	\|u\|_{L^p_{\omega}(\Omega)} := \left( \int_\Omega |u(x)|^p \omega(x)\, dx \right)^{\frac{1}{p}} < \infty \right\}.
	\]
	The space $L^p_{\omega}(\Omega)$ is a Banach space, and many of the standard properties 
	of classical Lebesgue spaces extend to this setting under suitable assumptions on the weight. 
	In particular, if $\omega$ is positive almost everywhere, then $L^p_{\omega}(\Omega)$ is separable. Note that when $\omega=1$, the space $L^p_{\omega}(\Omega)$ reduces to the classical Lebesgue space $L^p(\Omega)$.
	
	Based on this, we define the weighted Sobolev space $W_{\omega}^{1,p}(\Omega)$ as a function
	space consisting of all measurable functions $v \in L^p_{\omega}(\Omega) $  whose weak derivatives $Dv$ also belong to $L^p_{\omega}(\Omega) $. The norm of $W_{\omega}^{1,p}(\Omega)$ is given by
	$$
	\|u\|_{W_{\omega}^{1,p}(\Omega)} := \|u\|_{L^p_{\omega}(\Omega)} 
	+ \|D u\|_{L^p_{\omega}(\Omega)}.
	$$
	 Throughout this paper, we  further assume that $1<p<\infty$ and $\omega\in A_p$, which are natural assumptions for the weighted setting considered below. Then $W_{\omega}^{1,p}(\Omega)$ becomes a reflexive Banach space.  $W_{\omega,0}^{1,p}(\Omega)$ is defined as the closure of $C_c^\infty(\Omega)$ with respect to the norm of $W_{\omega}^{1,p}(\Omega)$.  We denote by $X_{\mathrm{loc}}(\Omega)$ the space of functions $u$ such that $u \in X(\Omega')$ for all $\Omega' \Subset \Omega$. These spaces naturally arise in the study of degenerate and non-uniformly elliptic problems, where the weight is intrinsically linked to the underlying structure of the equation.
	
	We introduce the notion of very weak solutions to the main equation \eqref{maineq}, which is natural in view of our goal of obtaining estimates below the natural exponent for the gradient. Recall that we denote $\omega(x) = |\mathbb{M}(x)|^p$.
	\begin{definition}[Very weak solution]\label{defvery}
		Suppose that $  |\mathbf{f} |^{p-1} \omega \in L^1_{\mathrm{loc}}(\Omega)$. A function $u \in W^{1,1}_{\mathrm{loc}}(\Omega)$ is called a very weak solution to \eqref{maineq} if $  |  D u  |^{p-1} \omega \in L^1_{\mathrm{loc}}(\Omega)$ and  for every $\eta \in C_c^\infty(\Omega)$,
		$$
		\int_\Omega A_{\mathbb{M}}(x,D u)\cdot D \eta \, dx
		=
		\int_\Omega A_{\mathbb{M}}(x,\mathbf{f})\cdot D \eta \, dx .
		$$
	\end{definition}
	
	 In particular, the above formulation is well-defined since both $A_{\mathbb{M}}(x,Du)$ and $A_{\mathbb{M}}(x,\mathbf{f})$ belong to $L^1_{\mathrm{loc}}(\Omega)$ under the stated assumptions. If $u$ is a very weak solution, the above equality extends to every compactly supported Lipschitz test function  $\eta$ by a standard weak-* approximation argument.
	 
	 Our main result below is a weighted Calderón–Zygmund estimate in the very weak regime in the measure-theoretic setting. The important point is that the exponent in the initial Sobolev regularity may lie strictly below the natural growth exponent $p$, while the conclusion reaches the whole neighborhood of $p$. \begin{theorem}\label{main1}
		Assume \eqref{struc} and \eqref{unidegen}. Suppose that $\omega \in A_{p}$ for some $p>1$. Then there exists
		$ \delta_0 = \delta_0(n,p,\nu,L,\Lambda,[\omega]_{A_{p}})>0$
		such that the following holds: Let $u \in W^{1,p-\delta_{0}}_{\omega,\mathrm{loc}}(\Omega)$ be a very weak solution of \eqref{maineq}. If $ \mathbf{f} \in L^\gamma_{\omega, \mathrm{loc}}(\Omega)$ for $\gamma \in [p-\delta_0,p+\delta_0]$, then $Du \in L^\gamma_{\omega, \mathrm{loc}}(\Omega)$, and for every ball $B_{2R} \Subset \Omega$ we have
		$$
		  \int_{B_{R}} | D u |^{\gamma} \, \omega\, dx
		\le C \left[ \omega(B_{2R})^{\frac{p-\delta_0-\gamma}{p-\delta_0}}
		\left( \int_{B_{2R}} | D u |^{p-\delta_0} \, \omega\, dx \right)^{\frac{\gamma}{p-\delta_0}}
		+ \int_{B_{2R}} | \mathbf{f} |^{\gamma} \, \omega\, dx
		\right],
		$$
		where $C$ depends only on $n,p, \nu,L,\Lambda$ and $[\omega]_{A_{p}}$.
	\end{theorem}
	Note that if $\delta_{0} \leq 1$, then  $u \in W^{1,p-\delta_{0}}_{\omega}(\Omega)$ implies $  |  D u  |^{p-1} \omega \in L^1_{\mathrm{loc}}(\Omega)$, which ensures the consistency of the definition of a very weak solution. Additionally, we recall that the Muckenhoupt class enjoys a self-improving property. By choosing $\delta_0>0$ sufficiently small so that $\omega \in A_{p-\delta_{0}}$, we have $u \in W^{1,1}_{\mathrm{loc}}(\Omega)$. Moreover, the standard weighted estimates including the Hardy--Littlewood maximal operator estimate and weighted Poincaré inequalities, can be applied in $W^{1,p-\delta_0}_\omega(\Omega)$.
	
	\begin{remark}
		The dependence of $\delta_0$ on $[\omega]_{A_p}$ appearing in Theorem \ref{main1} is natural in the weighted setting. Indeed, even for weak solutions, the constants in weighted Poincaré inequalities depend on the $A_p$-characteristic of the weight. Consequently, the higher integrability exponent obtained through Gehring's lemma also depends on $[\omega]_{A_p}$. Therefore, the dependence of $\delta_0$ on $[\omega]_{A_p}$ does not appear to be avoidable in general.
	\end{remark}
	\begin{remark}\label{multip}
		The multiplier formulation can also be treated within a weighted framework. More precisely, one may introduce the auxiliary weight  $\mu = |\mathbb{M}|^{\gamma-p}$ and consider the corresponding weighted gradient estimates, in which the relevant integrals are taken with respect to the measure $\mu \omega dx$.  Under the additional assumptions that $\omega \in A_{p}$ and $u \in W^{1,p-\delta_{0}}_{\omega^{1-\delta_{0}/p}}(\Omega)$ is a very weak solution, we expect the estimate
		$$
		\Xint-_{B_{R}} |\mathbb{M} D u |^{\gamma} \, dx
		\le C \left[
		\left( \Xint-_{B_{2R}} |\mathbb{M} D u |^{p-\delta_0}  \, dx \right)^{\frac{\gamma}{p-\delta_0}}
		+ \Xint-_{B_{2R}} |\mathbb{M} \mathbf{f} |^{\gamma} \, dx
		\right]
		$$
		to hold for every admissible ball $B_{R}$. We do not pursue such weighted estimates in this paper.
	\end{remark}
	
	\section{Preliminary Lemmas}\label{Sec3} 
	
	In this section, we introduce auxiliary tools that will be used to prove the main theorem. Most of these results are stated without proof, as they are standard or can be found in the literature.
	We first recall some standard estimates for the $V$-map. There exists a constant $c\geq 1$ depending only on $p$ and $n$ such that for all $\xi, \zeta \in \mr^n$,
	\[
	c^{-1}(|\xi|+|\zeta|)^{p-2}|\xi-\zeta|^2
	\leq
	|V(\xi)-V(\zeta)|^2
	\leq
	c(|\xi|+|\zeta|)^{p-2}|\xi-\zeta|^2.
	\]
	Moreover, using the above estimate, for every  $\varepsilon>0$, there exists a constant $c_{\varepsilon}$  such that
	\begin{align}\label{vpesti}
		|\xi-\zeta|^p
		\leq
		\varepsilon |\xi|^p
		+
		c_{\varepsilon} |V(\xi)-V(\zeta)|^2
	\end{align} 
	Here $c_\varepsilon$ depends on $\varepsilon$, $n$, and $p$. This can be proved by considering separately the cases $p>2$ and $1<p \leq 2$. A more general version of this inequality can be found in \cite{DSV12}, equation (2.10).
	
	 We next recall some properties of weights $\omega$ belonging to the Muckenhoupt class $A_{p}$.
	
\begin{lemma}\label{mucken}
	 Let $1<p<\infty$ and let $\omega \in A_p$. Then the following properties hold:
	\begin{enumerate}
		\item[(a)] (Doubling property) Let $E$ be a measurable subset of a ball $B$. Then
		\begin{eqnarray}\label{aprhproperty}
			\frac{1}{[\omega]_{A_{p}}}\left( \frac{|E|}{|B|} \right)^{p}  \leq \frac{\omega(E)}{\omega(B)}.
		\end{eqnarray} In particular, the weight \(\omega\) satisfies the doubling property
		\begin{equation}\label{double}
		\omega(B_{\rho})
		\leq \omega(B_{2\rho})
		\leq 2^{np}[\omega]_{A_{p}}\omega(B_{\rho}),
		\end{equation}
		for every ball \(B_{\rho}\subset \mathbb{R}^{n}\).
		\item[(b)] (Self-improving property) 
		There exists a positive constant $p_{0} \in (1,p)$, depending only on $n, p$ and $[\omega]_{A_{p}}$, such that
		$$
		 \omega \in A_{p_{0}}
		$$
		with $[\omega]_{A_{p_{0}}} \leq C(n, p, [\omega]_{A_{p}}).$	
	
		\item[(c)] Assume $\delta \in \left(0,\frac{p-p_{0}}{2}\right]$. Let $f$ be a non-trivial
		locally integrable function on $B_{\rho} \subset \mr^n$, extended $f$ by zero to
		$\mathbb{R}^{n}$. Then the weight $[\mathcal{M}(|f|)]^{-\delta} \omega$ is in the Muckenhoupt class $A_{p}$ with
		\begin{align}
			\notag [ [\mathcal{M}(|f|)]^{-\delta} \omega ]_{A_{p}} \leq c,
		\end{align}
		where the constant $c$ depends on  $n, p$ and $[\omega]_{A_{p}}$. 	Consequently, if $|f|^{p-\delta}\omega \in L^1(B_{\rho})$, then we have
		\begin{align}
			\notag \int_{\mr^n} [\mathcal{M}(|f|)(x)]^{p-\delta}\omega(x)\,dx
			\leq
			c\int_{B_{\rho}} |f(x)|^{p-\delta}\omega(x)\,dx.
		\end{align}
		for some positive constant $c$ depending on $n, p$ and $[\omega]_{A_{p}}$. 
	\end{enumerate}
\end{lemma}

	\begin{proof}
	The proofs of (a) and (b) of the lemma can be found in \cite{MP11} and \cite{S93}, respectively. Before proving (c), we note that if  $\epsilon \leq \frac{1}{2}$, then for any ball $B_{s}(z) \subset
	\mr^{n}$, we have
	\begin{align}\label{Coifman}
		 \Xint-_{B_{s}(z)} [\mathcal{M}(|f|)]^{\epsilon} \, dx \leq  c  [\mathcal{M} (|f|) (z)]^{\epsilon} \quad \textrm{a.e.} \  z\in \mr^{n} ,
	\end{align}
	where the constant $c$ depends only on $n$. This follows from the Coifman--Rochberg theorem; see, for instance, \cite{S93}, p.~214.  We next recall the definition of the $A_{p}$ class.
	\begin{align*}
		[[\mathcal{M}(|f|)]^{-\delta} \omega ]_{A_{p}}  = \sup_{B_{\rho} \subset \mathbb{R}^n}
		\left( \hspace{0.25em} \Xint-_{B_{\rho}} [\mathcal{M}(|f|)]^{-\delta} \omega \,dx \right)
		\left(  \hspace{0.25em} \Xint-_{B_{\rho}} [\mathcal{M}(|f|)]^{\frac{\delta}{p-1}} \omega^{-\frac{1}{p-1}}\,dx \right)^{p-1}  
	\end{align*}
	The first factor can be estimated as follows.
	\begin{align*}
		  \Xint-_{ B_{\rho}  }  [\mathcal{M}(|f|)]^{-\delta} \omega   \, dx  & \leq    \operatorname*{ess\,sup}_{B_\rho} [\mathcal{M}(|f|)]^{-\delta}  \hspace{0.25em} \Xint-_{ B_{\rho}  }   \omega   \, dx \\
		  & \leq    \left( \operatorname*{ess\,inf}_{B_\rho} [\mathcal{M}(|f|)]\right)^{-\delta} \hspace{0.25em}\Xint-_{ B_{\rho}  }   \omega   \, dx.
	\end{align*}
For the second factor, we use \eqref{Coifman} with $\epsilon = \frac{\delta}{p-p_{0}} \leq \frac{1}{2}$.
	\begin{align*}
		  & \Xint-_{B_{\rho}} [\mathcal{M}(|f|)]^{\frac{\delta}{p-1}} \omega^{-\frac{1}{p-1}}\,dx  \leq   	\left(  \hspace{0.25em} \Xint-_{B_{\rho}} [\mathcal{M}(|f|)]^{\frac{\delta}{p-p_{0}}} \,dx \right)^{\frac{p-p_{0}}{p-1}}	\left(  \hspace{0.25em} \Xint-_{B_{\rho}}  \omega^{-\frac{1}{p_{0}-1}}\,dx \right)^{\frac{p_{0}-1}{p-1}} \\
		&  \leq    \left(   \operatorname*{ess\,inf}_{z \in B_{\rho} }  2^{n}  \hspace{0.25em} \Xint-_{B_{2\rho}(z)} [\mathcal{M}(|f|)]^{\frac{\delta}{p-p_{0}}} \, dx   \right)^{\frac{p-p_{0}}{p-1}} 	\left(  \hspace{0.25em} \Xint-_{B_{\rho}}  \omega^{-\frac{1}{p_{0}-1}}\,dx \right)^{\frac{p_{0}-1}{p-1}} \\
		&  \leq  c_{0}  \operatorname*{ess\,inf}_{z \in B_{\rho} }  [\mathcal{M}(|f|)]^{\frac{\delta}{p-1}}	\left(  \hspace{0.25em} \Xint-_{B_{\rho}}  \omega^{-\frac{1}{p_{0}-1}}\,dx \right)^{\frac{p_{0}-1}{p-1}}
	\end{align*}
	for some constant $c_{0}$ depending only on $n$ and $p$. Combining the above estimates and taking the supremum over $B_{\rho} \subset \mr^n$, we obtain 
	 \begin{align*}
	 	  [[\mathcal{M}(|f|)]^{-\delta} \omega ]_{A_{p}}   \leq c_{0}^{p-1}   \sup_{B_{\rho} }	  \hspace{0.25em}\Xint-_{ B_{\rho}  }   \omega   \, dx	\left(  \hspace{0.25em} \Xint-_{B_{\rho}}  \omega^{-\frac{1}{p_{0}-1}}\,dx \right)^{p_{0}-1}   = c_{0}^{p-1} [\omega]_{A_{p_{0}}}. 
	 \end{align*}
	 Since $[\omega]_{A_{p_{0}}} $ depends only on $n, p$ and $[\omega]_{A_{p}}$, the proof of (c) is complete.
	 For the second estimate of (c), applying the weighted boundedness of the maximal operator gives
	\begin{align}
		\notag \int_{\mr^n} [\mathcal{M}(|f|)(x)]^{p-\delta}\omega(x)\,dx
		&= \int_{\mr^n} [\mathcal{M}(|f|)(x)]^{p}[\mathcal{M}(|f|)(x)]^{-\delta}\omega(x)\,dx \\
		\notag & \leq c\int_{\mr^n} |f(x)|^{p} [\mathcal{M}(|f|)(x)]^{-\delta}\omega(x) \,dx\\
		\notag & = c\int_{B_{\rho}} |f(x)|^{p} [\mathcal{M}(|f|)(x)]^{-\delta}\omega(x) \,dx\\
		\notag & \leq c\int_{B_{\rho}}  |f(x)|^{p-\delta}\omega(x) \,dx,
	\end{align}
	where we used $f(x)=0$ on $\mr^n \backslash B_{\rho}$ in the third line and the pointwise inequality $|f(x)| \leq \mathcal{M}(|f|)(x) $ in $B_{\rho}$ in the last line.
	\end{proof}
	\begin{remark}
		The same argument applies more generally if $[\mathcal{M}(|f|)]^{-\delta}$ is replaced by a weight $\mu$ satisfying
		\[
		\mu\in A_{1+p-p_0}
		\quad\text{and}\quad
		\operatorname*{ess\,sup}_{B}\mu
		\leq C \hspace{0.25em}\Xint-_{ B   }  \mu \, dx
		\]
		for every ball $B\subset\mathbb{R}^n$. In particular, under these assumptions, $\mu\omega\in A_p$ with a quantitative bound depending only on the corresponding weight characteristics. For the choice $\mu=[\mathcal{M}(|f|)]^{-\delta}$, the above conditions are known to hold; see, for instance, \cite[Lemma 4.6]{AM}.
	\end{remark}
	For the following weighted Sobolev-Poincaré inequality, we refer the reader to Theorem 1.2 in \cite{FKS82}. Here, we use the unweighted average as defined above. As will be clear from what follows, the choice between the weighted and unweighted averages does not affect the argument.   See also \cite{PR19}.
	\begin{lemma}\label{poincare} Suppose that $\omega \in A_{q}$ in $B_\rho \subset \mr^n$ for some $q$. Then there exists a constant $\theta=\theta(n,q, [\omega]_{A_{q}}) > \frac{n}{n-1} $  such that for any $ v \in W^{1, q}_{\omega}(B_\rho)$, there holds
		\begin{align}\label{poin1}
			\left( \frac{1}{\omega(B_{\rho})} \int_{B_{\rho}}  \left(\frac{|v-v_{B_{\rho}}|}{\rho}\right) ^{q\theta} \omega \, \mathrm{d}x \right)^{\frac{1}{\theta}}  \leq    \frac{c}{\omega(B_{\rho})}  \int_{B_{\rho}}  |Dv|^{q}   \omega \, \mathrm{d}x
		\end{align}
		for some positive constant $c$ depending only on $n,q$ and $[\omega]_{A_{q}}$.
	\end{lemma}
	The following Lipschitz truncation allows us to approximate Sobolev functions by Lipschitz functions with uniformly controlled gradients, which will be used as test functions in the equation. This is a standard technique in nonlinear PDEs; see, for instance, \cite{AF88,DMS08}.
	\begin{lemma}
		\label{liptrun}
	     For $v  \in W_{0}^{1,1}(B_{\rho} )$ with $B_{\rho} \subset \mr^n$, we write
		$$
		\quad E_{\lambda}:=  \{x \in B_{\rho} :   \mathcal{M}(|Dv|)(x) >\lambda
		\}  \quad (\lambda>0).
		$$
		Then there exist a Lipschitz function $v_\lambda \in W_{0}^{1,\infty} (B_{\rho} ) $
		and a positive constant $c$ depending on $n$ such that
		$$v_\lambda(x)=v(x),   \quad  Dv_\lambda(x)=Dv(x) \quad    \quad \mathrm{a.e.} \ x \in B_{\rho}  \backslash E_{\lambda}  $$
		and that the estimate
		$ |Dv_{\lambda}(x) |  \leq c(n)\lambda$
		holds for a.e. $ x \in B_{\rho} $.
	\end{lemma} 
	For a weight $\omega\in A_q \ (q>1)$ and a function $u\in W^{1,q}_{\omega,0}(\Omega)$, Hölder's inequality implies that $u\in W^{1,1}_{0}(\Omega)$, which allows us to apply the above lemma to $u$.
	
	Finally, we recall the following standard iteration lemma, which will be used in the proof of the main result.
	\begin{lemma}\label{algeb}\cite[Lemma 6.1]{G03}
		Let  $\phi : [\frac{R}{2}, R] \rightarrow [0, \infty)$ be a bounded non-negative function. Suppose that for every choice of $r_{1}$ and $r_{2}$ such that $\frac{R}{2} \leq r_{1} < r_{2} \leq R$, we have
		\begin{align}
			\notag \phi(r_{1}) \leq d \phi(r_{2}) + \frac{A}{(r_{2}-r_{1})^{\beta}}+ B
		\end{align}
		for some positive numbers $A, B, \beta >0$ and $d \in (0, 1)$. Then,
		\begin{align}
			\notag \phi(r_{1}) \leq  c \left( \frac{A}{R^{\beta}}+ B \right)
		\end{align}
		for some positive constant $c$ depending on $d$ and $\beta$.
		In particular, the constant $c$ continuously depends on
		$\beta$.
	\end{lemma}

	\section{Higher integrability and Solvability}\label{Sec4} 
In this section, we examine the higher integrability of very weak solutions and the existence of a very weak solution to a homogeneous equation, which will be key elements in comparison estimates. From now on, we use $c$ to represent generic constants depending on $n, p,  \nu, L, \Lambda$, and $[\omega]_{A_{p}}$, whose precise values may vary from line to line. Recall that $\omega(x) := |\mathbb{M}(x)|^p$. 

By the self-improving property of Muckenhoupt weights, see Lemma \ref{mucken} (b), we may assume that $\omega\in A_{p_0}$ for some $p_0<p$. Here, $p_0$ depends only on $n$ and $p,[\omega]_{A_{p}}$. From now on, $p_0$ will denote this exponent throughout the paper.

Consider the following homogeneous equation:
	\begin{equation}\label{homeq}
		- \mathrm{div\,}A_{\mathbb{M}}(x,Dw)=0 \quad  \textrm{in}\ \Omega.
	\end{equation}
	We establish the higher integrability result.
	
	\begin{lemma}\label{higher}
		
		Assume \eqref{struc} and \eqref{unidegen}. If  $\omega \in A_{p}$ for some $p>1$, then there exists a small constant $\sigma=\sigma(n, p, \nu, L, \Lambda, [\omega]_{A_{p}})$ such that
		every very weak solution $w\in W_{\omega}^{1,p-\sigma}  ( \Omega)$
		to the problem \eqref{homeq} belongs to $W_{\omega, \textrm{loc}}^{1,p+\sigma}(\Omega)$. Moreover, we
		have the following estimate
		\begin{equation}\label{higheresti}
			\left(  \frac{1}{\omega(B_{\rho})} \int_{B_{\rho}} |Dw|^{p+\sigma} \omega \, dx  \right)^{\frac{1}{p+\sigma}}  \leq c   \left(  \frac{1}{\omega(B_{2\rho})}  \int_{B_{2\rho}}|Dw|^{p-\sigma} \omega \, dx \right)^{\frac{1}{p-\sigma}}  
		\end{equation}
		for some constant $c>0$ depending on  $ n,p, \nu, L, \Lambda$ and $[\omega]_{A_{p}}$, whenever $B_{2\rho} \subset \Omega$.
	\end{lemma}

	\begin{proof}
		By replacing $p_{0}$ with a larger exponent if necessary, which is possible by the nesting property of the Muckenhoupt classes, we may assume that \[p_{0} >  \max \{ 1, p -  1,  \frac{(n-1)p}{n} \}=:p_{*}.\]
		 Take any positive number $\delta \leq \frac{p-p_{0}}{2}$ and let $w\in W_{\omega}^{1,p-\delta}  ( \Omega)$ be a solution to \eqref{homeq}. Choose any ball $B_{2r}(y) \subset B_{2\rho} \subset \Omega$ and a cut-off
		function $\eta \in C^{\infty}_{0}(B_{2r}(y))$ such that 
		$$ \chi_{B_{r}(y)} \leq \eta \leq \chi_{B_{2r}(y)} \quad \textnormal{and} \quad |D\eta| \leq \frac{2}{r}. $$ 
		Consider a function
		$$v :=(w-w_{B_{2r}(y) }) \eta  \in W_{\omega, 0}^{1,p-\delta}(B_{2r}(y)).$$
		and apply Lemma \ref{liptrun} to $v$. For any $\lambda>0$, we find a Lipschitz function $v_\lambda \in
		W_{0}^{1,\infty} (B_{2r}(y) )$ such that
		$$v_\lambda(x)=v(x),   \quad  Dv_\lambda(x)=Dv(x) \quad    \quad \mathrm{a.e.} \ x \in B_{2r}(y) \backslash E_{\lambda}  $$ and that
		$|Dv_{\lambda}|\leq c(n)\lambda$ for a.e. on $ B_{2r}(y)$, where the exceptional set $E_{\lambda}$ is defined by
		$$
		\quad E_{\lambda}:=  \{x \in B_{2r}(y) :   \mathcal{M}(|Dv|)(x)  >\lambda \}.
		$$
		Here, $Dv$ is understood as its zero extension outside the domain. Take $v_{\lambda}$  as a test function to the equation \eqref{homeq}. Monotonicity condition \eqref{struc} gives
		\begin{align}\label{i1esti1}
			\notag    & \int_{B_{2r}(y) \backslash E_{\lambda}}     \mathbb{M}(x)\, A\bigl(x, \mathbb{M}(x) Dw \bigr) \cdot Dv_{\lambda}     \, dx \\ 
		  &= - \int_{ E_{\lambda}}   \mathbb{M}(x)\, A\bigl(x, \mathbb{M}(x) Dw \bigr) \cdot Dv_{\lambda}     \, dx   \leq c \lambda\int_{ E_{\lambda}}   | Dw |^{p-1}  \omega \, dx .
		\end{align}
		Multiplying both sides of \eqref{i1esti1} by $\lambda^{-(1+\delta)}$  and integrating from $0$ to $\infty$ with respect to $\lambda$, we obtain
		\begin{align}\label{i1esti2}
			\notag  I_{1} &  : = \int_{0}^{\infty} \lambda^{-(1+\delta)} \left[ \int_{B_{2r}(y) \backslash E_{\lambda}}     \mathbb{M}(x)\, A\bigl(x, \mathbb{M}(x) Dw \bigr) \cdot Dv_{\lambda}     \, dx  \right] d\lambda \\
			\notag & \leq c  \int_{0}^{\infty} \lambda^{-\delta} \int_{ E_{\lambda}}     | Dw |^{p-1} \omega  \, dx    \,   d\lambda   \\
			\notag & \leq c \int_{B_{2r}(y)}  \left[ \int_{0}^{\mathcal{M}(|Dv|)} \lambda^{-\delta}   | Dw |^{p-1} \omega    \,   d\lambda \right]   \, dx  \\
			\notag  & \leq c \int_{B_{2r}(y)} \mathcal{M}(|Dv|)^{1-\delta}  | Dw |^{p-1} \omega  \, dx \\
			\notag  &  \leq  c \left(\int_{B_{2r}(y)} \mathcal{M}(|Dv|)^{p-\delta} \omega  \, dx + \int_{B_{2r}(y)}   | Dw |^{p-\delta} \omega  \, dx  \right)  \\
			\notag &  \leq  c \left(\int_{B_{2r}(y)}  |Dv|^{p-\delta} \omega  \, dx + \int_{B_{2r}(y)}   | Dw |^{p-\delta} \omega  \, dx  \right)  \\
			\notag & \leq c \left(  \int_{ B_{2r}(y)}  \frac{|w-w_{B_{2r}(y)} |^{p-\delta}}{r^{p-\delta}} \omega  \, dx  + \int_{ B_{2r}(y)}  | Dw |^{p-\delta} \omega     \, dx \right) \\
			& \leq c  \int_{ B_{2r}(y)}  | Dw |^{p-\delta} \omega   \, dx .
		\end{align}

		Next, we estimate $I_{1}$ from below. Split the domain $ B_{2r}(y) \backslash  B_{r}(y)$ into
		\begin{align} 
			\notag   D_{1} =\{x \in B_{2r}(y) \backslash B_{r}(y) :   \mathcal{M}(|Dv|)(x) \leq \delta \mathcal{M}(|Dw|\chi_{B_{2r}(y)})(x)   \}
		\end{align}
		and
		\begin{align} 
			\notag  D_{2} = \{x \in B_{2r}(y) \backslash  B_{r}(y) :   \mathcal{M}(|Dv|)(x) > \delta \mathcal{M}(|Dw|\chi_{B_{2r}(y)})(x)  \} .
		\end{align}
		
		Since $Dv_{\lambda}=Dv$ a.e. on $B_{2r}(y) \backslash E_{\lambda}$
		and $v=w-w_{B_{2r}(y)}$ a.e. on $B_{r}(y)$, we have
		\begin{align}\label{i1esti3}
			\notag     I_{1} & = \int_{ B_{2r}(y)} \int_{\mathcal{M}(|Dv|)}^{\infty}   \lambda^{-(1+\delta)}   \mathbb{M}(x)\, A\bigl(x, \mathbb{M}(x) Dw \bigr) \cdot Dv        \,   d\lambda   \, dx\\
			\notag & = \frac{1}{\delta} \int_{ B_{2r}(y)}   [\mathcal{M}(|Dv|)]^{-\delta} \mathbb{M}(x)\, A\bigl(x, \mathbb{M}(x) Dw \bigr) \cdot Dv  \,    dx \\ 
			\notag & \geq \frac{c \nu}{\delta} \int_{ B_{r}(y)}   [\mathcal{M}(|Dv|)]^{-\delta}   | Dw |^{p}  \omega   \,    dx      \\
			\notag & \hspace{5mm} + \frac{1}{\delta} \int_{ B_{2r}(y) \backslash B_{r}(y)  }   [\mathcal{M}(|Dv|)]^{-\delta} \mathbb{M}(x)\, A\bigl(x, \mathbb{M}(x) Dw \bigr) \cdot Dv    \,    dx \\
			\notag & \geq \frac{1}{\delta} \left( c_{0}  \int_{ B_{r}(y)}   [\mathcal{M}(|Dv|)]^{-\delta}   | Dw |^{p}  \omega   \,    dx     - L \int_{ D_{1}}     [\mathcal{M}(|Dv|)]^{-\delta}  |Dw|^{p-1} |Dv| \omega    \, dx  \right. \\
			\notag &  \hspace{5mm} \left. - 2L \int_{ D_{2}}  [\mathcal{M}(|Dv|)]^{-\delta}  |Dw|^{p-1} \left|\frac{w-w_{B_{2r}(y)}}{r}\right|  \omega   \, dx    \right)   \\
			& = :  \frac{1}{\delta}(I_{2} - I_{3} - I_{4}).
		\end{align}
		for some positive constant $c_{0}$ depending on $n,p, \nu $ and $\Lambda$.
		We now derive a lower bound for $I_{2}$. Since $[\mathcal{M}(|Dv|)]^{-\delta} \omega \in  A_{p}$ by Lemma \ref{mucken}, the boundedness of the Hardy--Littlewood maximal operator with respect to the weight $[\mathcal{M}(|Dv|)]^{-\delta} \omega $ implies that 
		\begin{align}
			\notag     I_{2}&=  c_{0} \int_{ B_{r}(y)}   {\mathcal{M}(|Dv|)}^{-\delta}   |Dw|^{p}  \,   \omega dx     \\
				\notag  &  \geq c \int_{ B_{r}(y)}  {\mathcal{M}(|Dv|)}^{-\delta}   \left[ \mathcal{M}( |Dw| \chi_{B_{r}(y)})\right]^{p}  \omega     \, dx .
		\end{align}
		We first obtain a pointwise upper bound for $\mathcal{M}(|Dv|)$. If $x \in B_{r/2}(y)$, then
		\begin{align}\label{mbound}
			\notag  &   \mathcal{M}(|Dv|)   \leq  \sup_{0 < s \leq r - |x-y| } \left(  \hspace{0.25em} \Xint-_{ B_{s}(x)  }  |Dv|   \, dz \right)  + \sup_{s > r - |x-y| } \left(  \hspace{0.25em} \Xint-_{ B_{s}(x)  } |Dv|   \, dz \right)  \\
			\notag & \leq     \mathcal{M}( |Dw| \chi_{B_{r}(y)})(x)  + c   \hspace{0.25em} \Xint-_{ B_{2r}(y) }  |Dv|   \, dz  \\
			&\leq     \mathcal{M}( |Dw| \chi_{B_{r}(y)})(x)  + c  \hspace{0.25em} \Xint-_{ B_{2r}(y) }  |Dw|   \, dz,
		\end{align}
		where we have used Lemma \ref{poincare} for the last inequality.  We split into two cases depending on which of the last two terms in \eqref{mbound} is larger, which leads to the following estimate.
		
		\begin{align}
			\notag  &   \left\{  \mathcal{M}(|Dw|\chi_{B_{r}(y)})\right\}^{p-\delta} \\
			\notag & \leq 2{\mathcal{M}(|Dv|)}^{-\delta}   \left[ \mathcal{M}( |Dw| \chi_{B_{r}(y)})\right]^{p}  + \left( c  \hspace{0.25em} \Xint-_{ B_{2r}(y) }  |Dw|   \, dz \right)^{p-\delta}
		\end{align}
		for  $x \in B_{r/2}(y)$. This estimate leads to
		\begin{align}\label{i2esti}
		 I_{2}	\notag &   \geq  c \int_{ B_{r}(y)}  {\mathcal{M}(|Dv|)}^{-\delta}   \left[ \mathcal{M}( |Dw| \chi_{B_{r}(y)})\right]^{p}  \omega     \, dx  \\
			\notag & \geq c \int_{ B_{r/2}(y)} \left\{  \mathcal{M}(|Dw|\chi_{B_{r}(y)})\right\}^{p-\delta} \omega \, dx - c \omega(B_{r/2}(y)) \left( c  \hspace{0.25em} \Xint-_{ B_{2r}(y) }  |Dw|   \, dz \right)^{p-\delta}\\
			& \geq c \int_{ B_{r/2}(y)}  |Dw|^{p-\delta}   \omega  \, dx - c   \omega(B_{2r}(y)) \left(  \frac{1}{ \omega(B_{2r}(y))}  \int_{ B_{2r}(y) } |Dw|^{p_{0}}  \omega  \, dx \right)^{\frac{p-\delta}{p_{0}}} ,
		\end{align} 
		where we used  $W^{1,1}$
		-Poincare first
		and then the $A_{p_{0}}$ Holder inequality. We have also used the doubling property \eqref{double} of $\omega$.
		 We then estimate $I_{3}$ as follows  using the definition of $D_{1}$:
		\begin{align}
			\notag   I_{3} &= L \int_{ D_{1}}  [\mathcal{M}(|Dv|)]^{-\delta}  |Dw|^{p-1}|Dv| \omega   \, dx    \\
			\notag
			&\leq c  \int_{ B_{2r}(y) \backslash B_{r}(y)}  [\mathcal{M}(|Dv|)]^{1-\delta}  |Dw|^{p-1}   \omega   \, dx  \\
			\notag  &   \leq  c \delta^{1-\delta}\int_{ B_{2r}(y)}  \left\{ \mathcal{M}(|Dw|\chi_{B_{2r}(y)}) \right\}^{ {p-\delta}} \omega  \, dx \\
			\notag &\leq   c \delta^{1-\delta}     \int_{ B_{2r}(y)}   |Dw|^{p-\delta}    \omega \, dx .
		\end{align}
		To estimate $I_{4}$, we observe that $
			|Dw(x)|\leq \mathcal{M}(|Dw|\chi_{B_{2r}(y)})  $
		  for $x \in B_{2r}(y)$. Then we find that
		\begin{align}\label{i4esti}
			\notag  I_{4}& \leq   2L  \int_{B_{2r}(y)}  [\mathcal{M}(|Dv|)]^{-\delta}  |Dw|^{p-1} \left|\frac{w-w_{B_{2r}(y)}}{r}\right|  \omega   \, dx  \\
			\notag &  \leq 2 \delta^{-\delta} L  \int_{ B_{2r}(y)}  [\mathcal{M}(|Dw|\chi_{B_{2r}(y)}) ]^{p-1-\delta}  \frac{|w-w_{B_{2r}(y)}|}{r}   \omega  \, dx \\
			\notag&  \leq \varepsilon \int_{ B_{2r}(y)} [\mathcal{M}(|Dw|\chi_{B_{2r}(y)}) ]^{p-\delta}  \omega \, dx  +c(\varepsilon) \int_{ B_{2r}(y)}  \frac{|w-w_{B_{2r}(y)}|^{p-\delta}}{r^{p-\delta}}   \omega  \, dx\\
			& \leq  \varepsilon \int_{ B_{2r}(y)}  |Dw|^{p-\delta}  \omega  \, dx + c(\varepsilon) \omega(B_{2r}(y)) \left(  \frac{1}{ \omega(B_{2r}(y))}  \int_{ B_{2r}(y) } |Dw|^{p_{0}}  \omega  \, dx \right)^{\frac{p-\delta}{p_{0}}},
		\end{align}
		where the weighted Sobolev–Poincaré inequality \eqref{poin1} was applied in the last line. Combining \eqref{i1esti2}, \eqref{i1esti3} and \eqref{i2esti}--\eqref{i4esti} with the doubling property \eqref{double} of $\omega$, we obtain
		\begin{align}
			\notag   \frac{1}{ \omega(B_{r/2}(y))}  \int_{ B_{r/2}(y)}  |Dw|^{p-\delta}  \omega  \, dx & \leq  c(\varepsilon)  \left(  \frac{1}{ \omega(B_{2r}(y))}  \int_{ B_{2r}(y) } |Dw|^{p_{0}}  \omega  \, dx \right)^{\frac{p-\delta}{p_{0}}}    \\
			\notag &   + \frac{ c\delta + \varepsilon + c(\varepsilon) \delta^{1-\delta}  }{ \omega(B_{2r}(y))} \int_{ B_{2r}(y)}  |Dw|^{p-\delta}  \omega  \, dx.
		\end{align}
		
	Choosing $\varepsilon$ and $\delta$ sufficiently small, the standard absorption argument(see, for instance \cite{G03}) with Lemma \ref{algeb} gives
		\begin{align} 
			\notag    \frac{1}{ \omega(B_{r}(y))}  \int_{ B_{r}(y)}  |Dw|^{p-\delta}   \omega \, dx & \leq  c  \left(  \frac{1}{ \omega(B_{2r}(y))}  \int_{ B_{2r}(y) } |Dw|^{p_{0}}  \omega  \, dx \right)^{\frac{p-\delta}{p_{0}}}.  
		\end{align}
	Since the constant $c$ in the above expression does not depend on $\delta$, the weighted version of Gehring's lemma \cite[Theorem 1.5]{K94}  yields
	\begin{align} 
		\left(  \frac{1}{ \omega(B_{\rho})}  \int_{ B_{\rho} } |Dw|^{p-\delta+\varepsilon_{0}}  \omega  \, dx \right)^{\frac{1}{p-\delta+\varepsilon_{0}}}   & \leq  c  \left(  \frac{1}{ \omega(B_{2\rho})}  \int_{ B_{2\rho} } |Dw|^{p-\delta}  \omega  \, dx \right)^{\frac{1}{p-\delta}}.  
	\end{align}
	for some $\varepsilon_{0}>0$ independent of $\delta$. Therefore, taking $\delta =  \sigma =: \min \{\frac{p-p_{0}}{2}, \frac{\varepsilon_{0}}{2} \}$, we obtain the desired estimate \eqref{higheresti}.
	\end{proof}

	We then derive an a priori estimate for the following Dirichlet problem:
	\begin{equation}\label{apreq}
		\begin{cases}
			-\mathrm{div\,}A_{\mathbb{M}}(x,Dw)=-\mathrm{div\,} A_{\mathbb{M}}(x,\mathbf{f})   & \textrm{in}\ B_{\rho} \\
			w\in w_{0}+ W^{1,p-\delta}_{\omega, 0}( B_{\rho})
		\end{cases}
	\end{equation}
	for $B_{\rho} \subset \mr^n$, which will be subsequently used to establish the existence result.

	\begin{lemma}\label{apriori} 	Assume \eqref{struc} and \eqref{unidegen}. If  $\omega \in A_{p}$ for some $p>1$, then there exists a small constant   $\delta_{1}=\delta_{1}(n,p,	\nu, L, \Lambda, [\omega]_{A_{p}})$ such that the following holds: For any $\delta \in (0,
		\delta_{1}]$, if $w \in	W^{1,p-\delta}_{\omega}( B_{\rho})$ is a very weak solution to \eqref{apreq} with $w_{0}
		\in W^{1,p-\delta}_{\omega}( B_{\rho})$ and $ \mathbf{f} \in
		L^{p-\delta}_{\omega}( B_{\rho})$, then there holds
		\begin{equation}\label{apresti}
			\int_{B_{\rho}} |Dw|^{p-\delta} \omega \, dx   \leq c \left[  \int_{B_{\rho}} |Dw_{0}|^{p-\delta} \omega  \, dx +\int_{B_{\rho}} |\mathbf{f}|^{p-\delta} \omega  \, dx \right] ,
		\end{equation}
		where $c>0$ depends on  $ n,p, \nu,L, \Lambda$ and $[\omega]_{A_{p}}$.
	\end{lemma}
	
	\begin{proof}
		Assume that  $\delta \in (0,\frac{p-p_{0}}{2} ]$ and let $w$ be the solution to \eqref{apreq}. Consider a function 
		\[v :=
		w-w_{0} \in W_{\omega, 0}^{1,p-\delta}( B_{\rho}).
		\]
		By applying Lemma \ref{liptrun}, we obtain a Lipschitz truncation $v_\lambda \in
		W_{0}^{1,\infty} (B_{\rho})$ such that 
		\[
		v_\lambda(x)=v(x),   \quad  Dv_\lambda(x)=Dv(x) \quad    \quad \mathrm{a.e.} \ x \in B_{\rho} \backslash E_{\lambda}, 
		 \]
		and that
		$|Dv_{\lambda}|\leq c(n)\lambda$ for a.e. on $  B_{\rho}$, where
		\[
		\quad E_{\lambda}:=  \{x \in B_{\rho} :   \mathcal{M}(|Dv|)(x) >\lambda \}.
		\]
	 Using $v_{\lambda}$ as a test function in \eqref{apreq}, we obtain
	 \begin{align}\label{ii1esti1}
	 	\notag   & \int_{B_{\rho} \backslash E_{\lambda}} \mathbb{M}(x)\, A\bigl(x, \mathbb{M}(x) Dw \bigr) \cdot Dv_{\lambda}   \, dx  = - \int_{ E_{\lambda}}  \mathbb{M}(x)\, A\bigl(x, \mathbb{M}(x) Dw \bigr) \cdot Dv_{\lambda}  \, dx \\
	 	\notag & \hspace{15mm}   + \int_{B_{\rho} \backslash E_{\lambda}}  \mathbb{M}(x)\, A\bigl(x, \mathbb{M}(x) \mathbf{f} \bigr) \cdot Dv_{\lambda}   \, dx + \int_{ E_{\lambda}}  \mathbb{M}(x)\, A\bigl(x, \mathbb{M}(x) \mathbf{f} \bigr) \cdot Dv_{\lambda}   \, dx \\
	  &\hspace{5mm} \leq c \left( \lambda \int_{ E_{\lambda}}  |Dw|^{p-1} \omega  \, dx +\int_{B_{\rho} \backslash E_{\lambda}}  |\mathbf{f}|^{p-1} | Dv| \omega \, dx +  \lambda\int_{ E_{\lambda}} |\mathbf{f}|^{p-1} \omega  \, dx\right).
	 \end{align}
		Multiplying \eqref{ii1esti1} by $\lambda^{-(1+\delta)}$ and then integrating
		with respect to $\lambda$ over $(0,\infty)$, we obtain
		\begin{align}
			\notag I_{1} &  : = \int_{0}^{\infty} \lambda^{-(1+\delta)} \int_{B_{\rho} \backslash E_{\lambda}}  \mathbb{M}(x)\, A\bigl(x, \mathbb{M}(x) Dw \bigr) \cdot Dv_{\lambda}   \, dx \,   d\lambda \\
			\notag & \leq c_{1} \left( \int_{B_{\rho}} \left[  \int_{0}^{\mathcal{M}(|Dv|)} \lambda^{-\delta} |Dw|^{p-1} \omega    \,   d\lambda  \right]  \, dx  \right.   \\
			\notag &  \hspace{10mm}   + \int_{B_{\rho}} \left[  \int_{\mathcal{M}(|Dv|)}^{\infty} \lambda^{-(1+\delta)}  |\mathbf{f}|^{p-1} | Dv|   \omega  \,   d\lambda  \right]  \, dx    \\
			\notag & \hspace{15mm} \left. + \int_{B_{\rho}}  \left[  \int_{0}^{\mathcal{M}(|Dv|)} \lambda^{-\delta}   |\mathbf{f}|^{p-1}  \omega \,   d\lambda  \right]  \, dx \right)   \\
		    \notag & =: c_{1} ( I_{2}+I_{3}+I_{4} ),
		\end{align}
		for some positive constant $c_{1}$ depending on $n,p, \nu, L $ and $\Lambda$. Applying Young's inequality yields the estimate 
		\begin{align}\label{ii1esti2}
			\notag  I_{2} & = \frac{1}{1-\delta} \int_{ B_{\rho}}  [ \mathcal{M}(|Dv|) ]^{1-\delta} |Dw|^{p-1} \omega      \, dx \\
			\notag & \leq c \left( \int_{ B_{\rho}}  [ \mathcal{M}(|Dv|)]^{p-\delta} \omega \, dx  + \int_{ B_{\rho}} |Dw|^{p-\delta}  \omega \, dx \right)\\
			\notag & \leq c \left( \int_{ B_{\rho}}    |Dv|^{p-\delta} \omega \, dx  + \int_{ B_{\rho}} |Dw|^{p-\delta}  \omega \, dx \right) \\
			& \leq c \left( \int_{ B_{\rho}}    |Dw|^{p-\delta} \omega \, dx  + \int_{ B_{\rho}} |Dw_{0}|^{p-\delta}  \omega \, dx \right) .
		\end{align} 
		An analogous argument gives the following estimate for $I_{3}$:
		\begin{align}\label{ii1esti3}
			\notag &  I_{3} = \frac{1}{\delta}\int_{B_{\rho}}  [ \mathcal{M}(|Dv|)]^{-\delta} |\mathbf{f}|^{p-1} | Dv|  \omega   \, dx  \leq \frac{c}{\delta}\int_{B_{\rho}}  [ \mathcal{M}(|Dv|)]^{1-\delta} |\mathbf{f}|^{p-1} \omega    \, dx \\
			\notag& \leq \frac{\varepsilon}{\delta}   \int_{ B_{\rho}} [ \mathcal{M}(|Dv|)]^{p-\delta}   \omega   \, dx   + \frac{c(\varepsilon )}{\delta}   \int_{ B_{\rho}} |\mathbf{f}|^{p-\delta}   \omega   \, dx  \\
			\notag & \leq \frac{c \varepsilon }{\delta}   \int_{ B_{\rho}}   |Dv|^{p-\delta}   \omega  \, dx   + \frac{c(\varepsilon )}{\delta}   \int_{ B_{\rho}} |\mathbf{f}|^{p-\delta}   \omega   \, dx   \\
			 & \leq  \frac{c \varepsilon }{\delta}   \int_{ B_{\rho}}   |Dw|^{p-\delta}   \omega  \, dx   + \frac{c(\varepsilon )}{\delta} \left[   \int_{ B_{\rho}}     |Dw_{0}|^{p-\delta}   \omega  \, dx+  \int_{ B_{\rho}}   |\mathbf{f}|^{p-\delta}   \omega    \, dx \right]  .
		\end{align}
		We estimate $I_{4}$ as follows:
		\begin{align}\label{ii1esti4}
			\notag   I_{4} & = \int_{B_{\rho}}  \left[ \int_{0}^{\mathcal{M}(|Dv|)} \lambda^{-\delta} |\mathbf{f}|^{p-1}  \omega \,   d\lambda  \right]   \, dx\\
			\notag &  \leq  c \int_{ B_{\rho}}  [\mathcal{M}(|Dv|)]^{1-\delta}  |\mathbf{f}|^{p-1}    \omega   \, dx \\
			\notag& \leq c \left(   \int_{ B_{\rho}}  [\mathcal{M}(|Dv|)]^{p-\delta}  \omega \, dx   +   \int_{ B_{\rho}}   |\mathbf{f}|^{p-\delta}    \omega  \, dx  \right)  \\
			&  \leq c   \left( \int_{ B_{\rho}}    |Dw|^{p-\delta} \omega \, dx +\int_{ B_{\rho}}     |Dw_0|^{p-\delta} \omega \, dx   +    \int_{ B_{\rho}}  |\mathbf{f}|^{p-\delta}  \omega  \, dx \right) .
		\end{align}
		We now derive a lower bound for $I_{1}$. Since
		\[
		Dv_{\lambda}=Dv
		\quad\text{on } B_{\rho}\setminus E_{\lambda},
		\]
		it follows that
		\begin{align}\label{ii1esti5}
			\notag  I_{1} &  = \int_{ B_{\rho}}  \left[  \int_{\mathcal{M}(|Dv|)}^{\infty} \lambda^{-(1+\delta)}  \mathbb{M}(x)\, A\bigl(x, \mathbb{M}(x) Dw \bigr) \cdot Dv_{\lambda}   \,   d\lambda \right]   \, dx\\
			\notag & = \frac{1}{\delta} \int_{B_{\rho}}   [\mathcal{M}(|Dv|)]^{-\delta}  \mathbb{M}(x)\, A\bigl(x, \mathbb{M}(x) Dw \bigr) \cdot Dv    \,    dx \\
			\notag & =  \frac{1}{\delta} \int_{B_{\rho}}   [\mathcal{M}(|Dv|)]^{-\delta} \mathbb{M}(x) \left( A\bigl(x, \mathbb{M}(x) Dw \bigr) -    A\bigl(x, \mathbb{M}(x) Dw_{0} \bigr) \right)\cdot Dv     \,    dx  \\
			\notag & \hspace{10mm} +   \frac{1}{\delta} \int_{B_{\rho}}  [\mathcal{M}(|Dv|)]^{-\delta}  \mathbb{M}(x)\, A\bigl(x, \mathbb{M}(x) Dw_{0} \bigr) \cdot Dv     \,    dx   \\
			\notag & \geq \frac{c_{2}}{\delta} \int_{ B_{\rho}}   [\mathcal{M}(|Dv|)]^{-\delta} |V(Dw)-V(Dw_{0})|^2 \omega \, dx   \\
			\notag & \hspace{10mm} - \frac{L}{\delta} \int_{ B_{\rho}}  [\mathcal{M}(|Dv|)]^{-\delta} |Dw_{0}|^{p-1} |Dv|  \omega   \, dx  \\
			& = :  \frac{1}{\delta}(I_{5} - I_{6}) 
		\end{align}
		for a constant $c_{2}$ depending on $n, p, \nu $ and $\Lambda$. For the estimate of $I_{5}$, we invoke Lemma~\ref{mucken} once more to infer that
	\[	
		[\mathcal{M}(|Dv|)]^{-\delta} \omega \in A_{p},
	\]
		and consequently,
		\begin{align}\label{ii2esti}
			\notag   I_{5} &= c_{2} \int_{ B_{\rho}}   [\mathcal{M}(|Dv|)]^{-\delta}  |V(Dw)-V(Dw_{0})|^2  \omega \,    dx \\
			& \geq c \int_{ B_{\rho}}  [\mathcal{M}(|Dv|)]^{-\delta} \left\{ \mathcal{M}( |V(Dw)-V(Dw_{0})|^{\frac{2}{p}}) \right\}^{p}   \omega   \, dx .
		\end{align}
		Also, we observe that for $x \in B_{\rho}$,
		\begin{align}\label{gbound2}
			 \mathcal{M}(|Dv|)(x) \leq  c \left[    \mathcal{M}( |V(Dw)-V(Dw_{0})|^{\frac{2}{p}}) (x)  + \mathcal{M}( | Dw_{0}|) (x) \right]  .
		\end{align}
		Considering separately the cases in which either of the last two terms in \eqref{gbound2} dominates, we obtain
		\begin{align}\label{algebraic2}
			\notag  &   [  \mathcal{M}( |V(Dw)-V(Dw_{0})|^{\frac{2}{p}})   ]^{p-\delta}    \\
			& \leq  c   [\mathcal{M}(|Dv|)]^{-\delta} [  \mathcal{M}( |V(Dw)-V(Dw_{0})|^{\frac{2}{p}})   ]^{p}     +  c [\mathcal{M}( | Dw_{0}|) ]^{p-\delta}
		\end{align}
		on $ B_{\rho}$. Putting together \eqref{ii2esti} and \eqref{algebraic2}, we arrive at
		\begin{align}\label{ii2esti2}
			\notag   I_{5} &  \geq c \int_{ B_{\rho}} \left\{  \mathcal{M}( |V(Dw)-V(Dw_{0})|^{\frac{2}{p}}) \right\}^{p-\delta}     \omega   \, dx - c \int_{ B_{\rho}}  [\mathcal{M}( | Dw_{0}|) ]^{p-\delta} \omega \, dx \\
			\notag &   \geq c \int_{ B_{\rho}}  |V(Dw)-V(Dw_{0})|^{\frac{2(p-\delta)}{p}}    \omega    \, dx -   c \int_{ B_{\rho}} |Dw_{0}|^{p-\delta} \omega \, dx \\
			& \geq c \int_{ B_{\rho}}  |Dw|^{p-\delta}   \omega  \, dx - c \int_{ B_{\rho}}  |Dw_{0}|^{p-\delta} \omega  \, dx .
		\end{align}
		By Young's inequality, the term $I_{6}$ can be estimated as follows:
		\begin{align}\label{ii3esti}
			\notag  I_{6}  &=  L \int_{ B_{\rho}}   [\mathcal{M}(|Dv|)]^{-\delta}   |Dw_{0}|^{p-1} |Dv|  \omega     \, dx \\
			\notag & \leq L \int_{ B_{\rho}}   [\mathcal{M}(|Dv|)]^{1-\delta} |Dw_{0}|^{p-1}     \omega  \, dx \\
			\notag & \leq \varepsilon \int_{ B_{\rho}} [\mathcal{M}(|Dv|)]^{p-\delta}   \omega  \, dx  + c(\varepsilon) \int_{ B_{\rho}} |Dw_{0}|^{p-\delta}   \omega  \, dx  \\
			\notag & \leq c \varepsilon \int_{ B_{\rho}} |Dv|^{p-\delta}  \omega   \, dx  + c(\varepsilon) \int_{ B_{\rho}}|Dw_{0}|^{p-\delta}  \omega    \, dx\\
			& \leq  c \varepsilon \int_{ B_{\rho}} |Dw|^{p-\delta}  \omega   \, dx  + c(\varepsilon) \int_{ B_{\rho}}|Dw_{0}|^{p-\delta}  \omega    \, dx.
		\end{align}
		We finally combine \eqref{ii1esti1}--\eqref{ii1esti5}, \eqref{ii2esti2}   and
		\eqref{ii3esti}  to conclude
		\begin{align} 
			\notag  \int_{ B_{\rho}}  |Dw|^{p-\delta} \omega   \, dx  & \leq c \{ \varepsilon+\delta \} \int_{ B_{\rho}} |Dw|^{p-\delta} \omega    \, dx   \\
			\notag  & \hspace{5mm} + c \{ c(\varepsilon)+\delta \} \left[ \int_{ B_{\rho}}|Dw_{0}|^{p-\delta} \omega   \, dx   +\int_{ B_{\rho}}    |\mathbf{f}|^{p-\delta} \omega    \, dx  \right].
		\end{align}
		  By choosing $\varepsilon$ and $\delta_{1}$ sufficiently small, we obtain the desired estimate \eqref{apresti} for $\delta \leq \delta_{1}$.
	\end{proof}
	
	It should be emphasized that the estimate obtained above is only an a priori
	estimate at the level of the solution space itself; namely, it provides a
	$W^{1,p-\delta}(B_{\rho})$ estimate only for very weak solutions belonging to
	$W^{1,p-\delta}(B_{\rho})$. This is not yet the desired gradient estimate.
	Our remaining task is to establish a genuine higher integrability estimate,
	namely a $W^{1,q}$ estimate for
	$W^{1,p-\delta_{0}}(B_{\rho})$ solutions, where
	$q$ lies in the appropriate range $[p-\delta_{0}, p+\delta_{0}]$.
	
	 We next establish the existence of weak solutions to the homogeneous Dirichlet problem in $W^{1,p-\delta}(B_{\rho})$, which  will be used in the construction of comparison
	 estimates in the next section. To this end, we assume
	 that  $ 0<\delta \leq \delta_{2}:=\min\{\sigma,\delta_{1}\}, $
	 where $\sigma$ and $\delta_{1}$ denote the constants appearing in
	 Lemmas~\ref{higher} and \ref{apriori}, respectively, so that both
	 results are available throughout the proof. Moreover, we have $p - \delta \geq p_{0} > 1 $ by choice of $\sigma$ and $\delta_{1}$.   See also
	 \cite[Section~7]{IS94}.

	\begin{corollary}\label{existence}
		 Assume \eqref{struc} and \eqref{unidegen}. If  $\omega \in A_{p}$ for some $p>1$,  then for all
		$w_{0}\in W^{1,p-\delta }_{\omega}  ( B_{\rho}) $ with $\delta \in
		(0,\delta_{2}]$, there exists a very weak solution $w\in
		W^{1,p-\delta }_{\omega}  ( B_{\rho})$ to the problem
		\begin{align}\label{exieq}
			\begin{cases}
				-\mathrm{div\,}A_{\mathbb{M}}(x,Dw)=0 & \textrm{in}\ B_{\rho} \\
				w\in w_{0}+ W^{1,p-\delta }_{\omega,0}( B_{\rho}),
			\end{cases}
		\end{align}
		which satisfies the estimate
		\begin{align}\label{exiesti}
			\int_{B_{\rho}}   |Dw|^{p-\delta} \omega \, dx   \leq c   \int_{B_{\rho}} |Dw_{0}|^{p-\delta} \omega  \, dx
		\end{align}
		where the  constant $c>0$ depends on  $ n, p, \nu, L, \Lambda $ and $ [\omega]_{A_{p}}$.
	\end{corollary}
	
	\begin{proof}
		Let $\delta\in(0,\delta_{2}]$, and let
		$\{w_{0,k}\}\subset C^{\infty}(\overline{B_{\rho}})$ be a sequence satisfying
		\begin{align} 
		 \notag	w_{0,k} \rightarrow w_{0} \quad \textrm{strongly in} \ W^{1,p-\delta}_{\omega}  ( B_{\rho}).
		\end{align}
		For each $k\in\mathbb{N}$, there exists a unique weak solution
		$w_{k}\in W^{1,p }_{\omega}( B_{\rho})$ to the Dirichlet problem
		\begin{align} 
		 \notag	\begin{cases}
				-\mathrm{div\,}A_{\mathbb{M}}(x,Dw_{k})=0 & \textrm{in}\ B_{\rho} \\
				w_{k}\in w_{0,k}+ W^{1,p }_{\omega,0}( B_{\rho}).
			\end{cases}
		\end{align}
		The existence of a weak solution follows from the standard existence theory for monotone operators, together with the fact that $|\mathbb{M}|^{p}$ belongs to the Muckenhoupt class $A_{p}$; see, for instance, \cite{S97, IM19}. By Lemma~\ref{apriori}, the sequence $\{w_k\}$ satisfies the uniform
		estimate
		\begin{align}\label{seqbound}
			\int_{B_{\rho}} |Dw_{k}|^{p-\delta} \omega  \, dx   \leq c   \int_{B_{\rho}} |Dw_{0,k}|^{p-\delta} \omega  \, dx  \leq  2 c   \int_{B_{\rho}} |Dw_{0}|^{p-\delta} \omega  \, dx 
		\end{align} 
		for all sufficiently large $k$.  Combining the estimate \eqref{seqbound} with Lemma~\ref{poincare}, we
		obtain a uniform bound for the sequence $\{w_k\}$ in
		$W^{1,p-\delta}_{\omega}(B_{\rho})$. Since $\omega \in A_{p-\delta}$ and $B_{\rho}$ is a bounded Lipschitz domain, the weighted Rellich–Kondrachov compact embedding theorem (see, for instance, \cite{GU09}) gives 
		$W^{1,p-\delta}_{\omega}(B_{\rho})\Subset L^{p-\delta}_{\omega}(B_{\rho})$. Therefore, there exists a subsequence, still
		denoted by $\{w_k\}$, and a function
		$w\in W^{1,p-\delta}_{\omega}(B_{\rho})$ such that
		\begin{align} 
		 \notag	w_{k} \rightharpoonup w \quad \textrm{weakly in} \ W^{1,p-\delta}_{\omega}  ( B_{\rho}),  \quad w_{k} \rightarrow w \quad \textrm{strongly in} \ L^{p-\delta}_{\omega} (B_{\rho}).
		\end{align}
		Moreover, since
		$$
		\tilde{w}_{k}:=w_{k}-w_{0,k}\in W^{1,p-\delta}_{\omega,0}(B_{\rho}),
		$$
		and $\tilde{w}_{k}$ converges weakly to $w-w_{0}$ as $k\to\infty$, we conclude that  $w\in w_{0}+ W^{1,p-\delta}_{\omega, 0}( B_{\rho})$.
		
		Fix $\phi\in C_{0}^{\infty}(B_{\rho})$ and let
		$U\Subset B_{\rho}$ be an open set such that
		$\operatorname{supp}\phi\subset U$. The higher integrability estimate in
		Lemma~\ref{higher} implies that $\{w_k\}$ is uniformly bounded in
		$W^{1,p}_{\omega}(U)$. Therefore, up to a subsequence, we have
		\begin{align} 
			\notag w_{k} \rightharpoonup w \quad \textrm{weakly in} \ W^{1,p}_{\omega} (U), \quad w_{k} \rightarrow w \quad \textrm{strongly in} \ L^{p}_{\omega} (U).
		\end{align}
		Using the Minty--Browder argument, we can pass to the limit in the weak
		formulation and conclude that $w$ is a weak solution of
		\[
		\mathrm{div}\,A_{\mathbb{M}}(x,Dw)=0
		\]
		in $U$. In particular, for every $\phi\in C_{0}^{\infty}(U)$, we have
		\begin{align} 
			\notag
			0
			&=
			\int_{U}\mathbb{M}(x)\,
			A\bigl(x,\mathbb{M}(x)Dw\bigr)\cdot D\phi\,dx  \\
			\notag&=
			\int_{B_{\rho}}\mathbb{M}(x)\,
			A\bigl(x,\mathbb{M}(x)Dw\bigr)\cdot D\phi\,dx .
		\end{align}
		Since $\phi$ was chosen arbitrarily, $w$ is a very weak solution of
		\eqref{exieq} with the estimates \eqref{exiesti}.
	\end{proof}

	\section{Proof of Theorem \ref{main1}}\label{Sec5} 
	This section is devoted to the proof of Theorem~\ref{main1}. To this end, we compare a solution $u$ to the	equation \eqref{maineq}
	$$
	-\operatorname{div} A_{\mathbb{M}}(x,D u)
	=
	-\operatorname{div} A_{\mathbb{M}}(x, \mathbf{f})
	\qquad \text{in } \Omega
	$$
	with a solution $w$ to the following homogeneous problem:
	\begin{align}\label{refeq}
		\begin{cases}
			\mathrm{div\,}A_{\mathbb{M}}(x,Dw)=0 & \textrm{in}\ B_{2r} \\
			w\in u+ W^{1,p-\delta_{0}}_{\omega,0}( B_{2r})
		\end{cases}
	\end{align}
	for $B_{2r} \subset \Omega $. The problem \eqref{refeq} will serve as the reference problem, where $\delta_{0}>0$ will be chosen later. We further assume that
	\begin{equation}\label{intlambda}
		\frac{1}{\omega(B_{2r})}\int_{B_{2r}} |Du|^{p-\delta_{0}} \omega \, dx  \leq  \Gamma^{p-\delta_{0}},  \quad   \frac{1}{\omega(B_{2r})} \int_{B_{2r}}  |\mathbf{f}|^{p-\delta_{0}} \omega \, dx \leq (\delta_{0} \Gamma)^{p-\delta_{0}}
	\end{equation}
	for some $\Gamma>0$. These assumptions will be imposed in the course of a level set argument developed later. We will frequently use the universal constants introduced in the previous
	section, such as $p_{0}, \sigma$ and $\delta_{2}$, without further mention. We now establish the comparison estimate between the solutions $u$ and $w$ introduced above.

	\begin{lemma}\label{compari}
	 Assume \eqref{struc}, \eqref{unidegen} and $\omega \in A_{p}$ with $p>1$. Then for any
		$\varepsilon \in (0,1]$, there exists a constant $ \delta_{0} = \delta_{0}(n, p, \nu, L, \Lambda, [\omega]_{A_{p}},  \varepsilon ) $ $ \in (0, \frac{\sigma}{2}]$
		such that the following holds: For any very weak solution $u \in
		W^{1,p-\delta_{0}}_{\omega} (\Omega)$ to \eqref{maineq} with
		$\mathbf{f} \in L^{p-\delta_{0}}_{\omega}(\Omega)$, if
		\eqref{intlambda} holds, then there exists a very weak solution $w
		\in W^{1,p-\delta_{0}}_{\omega} ( B_{2r}  ) $ to the equation
		\eqref{refeq} such that
		\begin{align}\label{compa}
			\frac{1}{\omega(B_{2r})} \int_{B_{2r}}|V(Du)-V(Dw)|^{2(1-\frac{\delta_{0}}{p})} \omega \, dx \leq \varepsilon \Gamma^{p-\delta_{0}}
		\end{align}
		with the estimate
		\begin{align}\label{compenergy}
			\frac{1}{\omega(B_{2r})}\int_{B_{2r}} |Dw|^{p-\delta_{0}} \omega \, dx  \leq  c \Gamma^{p-\delta_{0}} ,  \quad    \frac{1}{\omega(B_{r})} \int_{B_{r}} |Dw|^{p+\sigma} \omega \, dx     \leq  c \Gamma^{p+\sigma},
		\end{align}
		where the constant $c>0$ depends only on  $ n, p, \nu, L, \Lambda$ and $ [\omega]_{A_{p}}$.
	\end{lemma}

	\begin{proof}
		Let $0<\delta_{0}\leq \frac{\delta_{2}}{2}$. Then all the results established in Section 4 apply with $\delta=\delta_{0}$. By Corollary \ref{existence}, there exists a very weak solution
		$$
		w\in W^{1,p-\delta_{0}}_{\omega}(B_{2r})
		$$
		to \eqref{refeq} satisfying the first estimate in \eqref{compenergy}. Furthermore, by Lemma \ref{higher} and Hölder's inequality,
		\begin{align}
			\notag     \frac{1}{\omega(B_{r})} \int_{B_{r}} |Dw|^{p+\sigma} \omega \, dx    & \leq   c^{p+\sigma}  \left(  \frac{1}{\omega(B_{2r})} \int_{B_{2r}}   |Dw|^{p-\sigma} \omega \, dx  \right)^{\frac{p+\sigma}{p-\sigma}} \\
			\notag & \leq  c \left( \frac{1}{\omega(B_{2r})}  \int_{B_{2r}} |Dw|^{p-\delta_{0}}\omega  \, dx  \right)^{\frac{p+\sigma}{p-\delta_{0}}}  \leq  c \Gamma^{p+\sigma}.
		\end{align}
		This proves the second estimate in \eqref{compenergy}.
		
	   To establish \eqref{compa}, we set$$
	   v:=u-w\in W^{1,p-\delta_{0}}_{\omega,0}(B_{2r}),$$
	   and apply the Lipschitz truncation technique from Lemma \ref{liptrun} to obtain a function $v_\lambda\in W_{0}^{1,\infty}(B_{2r})$  for some $\lambda>0$. We have
	   $$
	   v_\lambda=v \quad \text{and} \quad Dv_\lambda=Dv
	   \quad\text{a.e. in } B_{2r}\setminus E_\lambda,
	   $$
	   where
	   $$
	   E_\lambda
	   := \left\{x\in B_{2r}:\mathcal{M}(|Dv|)(x)>\lambda \right\},
	   $$
	   and $ |Dv_\lambda|\le c\lambda $  a.e. in  $B_{2r}$.
	   Using \(v_\lambda\) as a test function in both \eqref{maineq} and \eqref{refeq}, we obtain
		\begin{align}\label{iii1esti1}
			\notag &   \int_{B_{2r}  \backslash E_{\lambda}}   \mathbb{M}(x) (  A\bigl(x, \mathbb{M}(x) Du \bigr) -  A\bigl(x, \mathbb{M}(x) Dw \bigr)  )\cdot Dv_{\lambda}  \, dx \\
			\notag&= - \int_{ E_{\lambda}}  \mathbb{M}(x) (  A\bigl(x, \mathbb{M}(x) Du \bigr) -  A\bigl(x, \mathbb{M}(x) Dw \bigr)  )\cdot Dv_{\lambda}  \, dx \\
			\notag& \hspace{5mm}+ \int_{ B_{2r} }  \mathbb{M}(x) (   A\bigl(x, \mathbb{M}(x) \mathbf{f} \bigr)  )\cdot Dv_{\lambda}   \, dx \\
			\notag &\leq c \left(  \lambda \int_{ E_{\lambda}}   |Du|^{p-1} \omega   \, dx +  \lambda \int_{ E_{\lambda}}  |Dw|^{p-1} \omega    \, dx \right. \\
		& \hspace{5mm} \left. +\int_{B_{2r} \backslash E_{\lambda}}   |\mathbf{f}|^{p-1}  | Dv| \omega \, dx +   \lambda \int_{ E_{\lambda}} |\mathbf{f}|^{p-1}  \omega  \, dx\right).
		\end{align}
	Multiplying \eqref{iii1esti1} by $\lambda^{-(1+\delta_{0})}$ and
	integrating with respect to $\lambda$ over $(0,\infty)$, we proceed as
	in the derivation of \eqref{ii1esti1} to obtain
		\begin{align}\label{iii1esti2}
			\notag &   \delta_{0} I_{1} : = \delta_{0} \int_{0}^{\infty} \lambda^{-(1+ \delta_{0})} \int_{B_{2r} \backslash E_{\lambda}}  (A_{\mathbb{M}}(x,Du)  -  A_{\mathbb{M}}(x,Dw)  )\cdot Dv_{\lambda}   \, dx \,   d\lambda \\
			\notag & \hspace{5mm} \leq c \{ \varepsilon_{1} +\delta_{0} \}   \left[    \int_{ B_{2r} }    |Du|^{p-\delta_{0}} \omega  \, dx  + \int_{ B_{2r} }    |Dw|^{p-\delta_{0}}  \omega \, dx  \right] \\
			\notag   &\hspace{10mm} + c \{ c( \varepsilon_{1}) +\delta_{0} \}   \int_{ B_{2r} }   |\mathbf{f}|^{p-\delta_{0}}   \omega   \, dx   \\
			 & \hspace{5mm} \leq c \{ \varepsilon_{1}+ c( \varepsilon_{1})\delta_{0} \}  \omega(B_{2r}) \Gamma^{p-\delta_{0}}.
		\end{align}
		Note that the term \(\delta_{0}^{p-\delta_{0}}\) appearing in \eqref{intlambda}
		can be absorbed into \(\delta_{0}\).  Let's find a lower bound of $I_{1}$. By Lemma \ref{mucken}, we again obtain
		\[
		[\mathcal{M}(|Dv|)]^{-\delta_{0}}\omega \in A_p.
		\]
		Consequently, it follows that
		\begin{align}\label{iii1esti3}
			\notag &   \hspace{-4mm} I_{1}  = \int_{ B_{2r} }  \int_{\mathcal{M}(|Dv|)}^{\infty} \lambda^{-(1+\delta_{0})}(A_{\mathbb{M}}(x,Du)  -  A_{\mathbb{M}}(x,Dw)  )\cdot Dv     \,   d\lambda   \, dx\\
			\notag & = \frac{1}{\delta_{0}} \int_{B_{2r} }   [\mathcal{M}(|Dv|)]^{-\delta_{0}} (A_{\mathbb{M}}(x,Du)  -  A_{\mathbb{M}}(x,Dw)  )\cdot Dv     \,    dx \\
			\notag  & \geq \frac{\nu}{\delta_{0}} \int_{ B_{2r} }   [\mathcal{M}(|Dv|)]^{-\delta_{0}} |V(Du)-V(Dw)|^2  \omega \, dx \\
			& \geq \frac{c}{\delta_{0}} \int_{ B_{2r} }  [\mathcal{M}(|Dv|)]^{-\delta_{0}} \left\{ \mathcal{M}( |V(Du)-V(Dw)|^{\frac{2}{p}})\right\}^{p}   \omega   \, dx.
		\end{align}
	 Using \eqref{vpesti}, for any $x \in B_{2r}$ and $\tilde{\varepsilon} \in (0,1]$, we obtain
		\begin{align}\label{gbound3}
			\notag  &   \mathcal{M}(|Dv|)(x) \leq   \mathcal{M}( \left[ c( \tilde{\varepsilon} ) |V(Du)-V(Dw)|^{\frac{2}{p}}+ \tilde{\varepsilon}|Du| \right] )(x)  \\
			&  \leq  c ( \tilde{\varepsilon} )  \mathcal{M}( |V(Du)-V(Dw)|^{\frac{2}{p}}) (x)    +  c \tilde{\varepsilon} \mathcal{M}(|Du|)(x).
		\end{align}
		Considering separately the cases in which either of the last two terms in \eqref{gbound3} dominates, we obtain
		\begin{align}
			\notag  &  \left\{ \mathcal{M}( |V(Du)-V(Dw)|^{\frac{2}{p}})\right\}^{p-\delta_{0}}    \\
			\notag  & \hspace{5mm} \leq  c( \varepsilon_{2} )  [\mathcal{M}(|Dv|)(x)]^{-\delta_{0}} \left\{ \mathcal{M}( |V(Du)-V(Dw)|^{\frac{2}{p}}) \right\}^{p}    +  \varepsilon_{2} [\mathcal{M}(|Du|)]^{p-\delta_{0}}
		\end{align}
		for any $\varepsilon_{2} \in (0,1]$. Recalling \eqref{iii1esti3}, we have
		\begin{align}\label{iii2esti2}
			\notag & \hspace{-4mm} \delta_{0} I_{1}  \geq \frac{1}{ c(\varepsilon_{2})}\int_{B_{2r} }   \left\{ \mathcal{M}( |V(Du)-V(Dw)|^{\frac{2}{p}})\right\}^{p-\delta_{0}}  \omega       \, dx \\
			\notag &\hspace{10mm}  - \frac{  \varepsilon_{2} }{ c(\varepsilon_{2})}   \int_{ B_{2r} } [\mathcal{M}(|Du|)]^{p-\delta_{0}} \omega \, dx \\
			\notag  &   \geq \frac{1}{ c(\varepsilon_{2})} \int_{ B_{2r} }  |V(Du)-V(Dw)|^{2(1-\frac{\delta_{0}}{p})}  \omega      \, dx -   \frac{   \varepsilon_{2} }{ c(\varepsilon_{2})}  \int_{ B_{2r} } |Du|^{p-\delta_{0}} \omega  \, dx \\
			&   \geq \frac{1}{ c(\varepsilon_{2})} \int_{ B_{2r} }  |V(Du)-V(Dw)|^{2(1-\frac{\delta_{0}}{p})}  \omega      \, dx -   \frac{  \varepsilon_{2} }{ c(\varepsilon_{2})}  \omega(B_{2r})\Gamma^{p-\delta_{0}} .
		\end{align}
		Combining \eqref{intlambda},\eqref{iii1esti2} \eqref{iii1esti3}, and \eqref{iii2esti2}, we obtain
		\begin{align} 
			\notag  \int_{ B_{2r} }  |V(Du)-V(Dw)|^{2(1-\frac{\delta_{0}}{p})}   \omega  \, dx &\leq  (   \varepsilon_{2}+ c (\varepsilon_{2})  \varepsilon_{1} +  c (\varepsilon_{1}) c( \varepsilon_{2})\delta_{0} ) \omega(B_{2r}) \Gamma^{p-\delta_{0}}  \\
			\notag &\leq \varepsilon  \omega(B_{2r}) \Gamma^{p-\delta_{0}} 
		\end{align}
		where, for the last inequality, we have chosen $\varepsilon_{2} = \frac{\varepsilon}{3}$, $\varepsilon_{1}
		=\frac{\varepsilon}{3c (\varepsilon_{2})} $ and $ \delta_{0}
		:=\frac{1}{3}\min{\{ \frac{\varepsilon}{c (\varepsilon_{1})c (\varepsilon_{2})}
			, \frac{\delta_{2}}{2}\}}$. This completes
		the proof.
	\end{proof}

	We are now ready to establish the main result of the paper.

	\begin{proof}[Proof of Theorem \ref{main1}]  Our argument follows the approach developed in \cite{AM07}, while its implementation in the present degenerate setting requires a more delicate control of the associated weight and its interaction with the nonlinear structure of the equation. We work under the same assumptions as in Theorem \ref{main1}. 
		In particular, $\delta_{0}$ is chosen according to Lemma \ref{compari}. 
		Notice that, if $\varepsilon$ in Lemma \ref{compari} is selected depending 
		only on $n,p,\nu,L,\Lambda$ and $[\omega]_{A_p}$, then the corresponding 
		choice of $\delta_{0}$ also depends only on 
		$n,p,\nu,L,\Lambda$ and $[\omega]_{A_p}$. We fix an arbitrary ball $B_{R}(z) \subset \Omega$ and introduce the
		following upper level sets:
		\[
		E^{s}_{\Gamma}:=\left\{x\in B_{s}(z): |Du|>\Gamma\right\},
		\qquad \Gamma>0,
		\]
		for every $s$ satisfying $\frac{R}{2}\leq s\leq R$. Let
		$B_{r_{1}}(z)$ and $B_{r_{2}}(z)$ be two concentric balls such that
		\[
		\frac{R}{2}\leq r_{1}<r_{2}\leq R .
		\]
	For each point \(y\in E^{r_{1}}_{\Gamma}\), we define a continuous function
	\(\Phi_y:(0,r_{2}-r_{1}]\rightarrow [0,\infty)\) by
	\begin{align} 
		\notag \Phi_y(\rho):=
		\frac{1}{\omega(B_{\rho}(y))}
		\int_{B_{\rho}(y)}
		\left(
		|Du|^{p-\delta_{0}}
		+\left(\frac{|\mathbf{f}|}{\delta_{0}}\right)^{p-\delta_{0}}
		\right)\omega\,dx .
	\end{align}
	Consequently, by the Lebesgue differentiation theorem, the function \(\Phi_y\) satisfies
	\begin{align} 
		\notag \lim_{\rho\to0}\Phi_y(\rho)
		=
		|Du(y)|^{p-\delta_{0}} + 
		\left(\frac{|\mathbf{f}|}{\delta_{0}}\right)^{p-\delta_{0}}
		\geq |Du(y)|^{p-\delta_{0}}
		>\Gamma^{p-\delta_{0}}
	\end{align}
	for \(\omega\,dx\)-almost every \(y\in E^{r_{1}}_{\Gamma}\), where the last
	inequality follows from the definition of \(E^{r_{1}}_{\Gamma}\). Since $\omega \in A_{p}$, 
	\(\omega>0\) almost everywhere and the measures \(\omega\,dx\) and \(dx\) have the
	same null sets. Therefore, the above conclusion also holds for
	\(dx\)-almost every \(y\in E^{r_{1}}_{\Gamma}\).
		Note that if $ \frac{r_{2}-r_{1} }{10} \leq \rho \leq r_{2}-r_{1}  $, then
		\begin{align}\label{lambdazero}
			\notag \Phi_{y}(\rho) & \leq  \frac{1}{\omega(B_{\frac{r_{2}-r_{1} }{10}}(y))} \int_{B_{r_{2}}(z)}  \left\{ 
			|Du|^{p-\delta_{0}}
			 + 	\left(\frac{|\mathbf{f}|}{\delta_{0}}\right)^{p-\delta_{0}} \right\} \omega \, dx \\ 
			 & \leq  \frac{[\omega]_{A_{p}}10^{np} r_{2}^{np}}{(r_2-r_1)^{np}\omega(B_{r_{2}}(z))} \int_{B_{r_{2}}(z)}  \left\{ 
			 |Du|^{p-\delta_{0}}
			 + 	\left(\frac{|\mathbf{f}|}{\delta_{0}}\right)^{p-\delta_{0}} \right\} \omega \, dx =: \Gamma_{0}^{p-\delta_{0}},
		\end{align}
		where we used the inequality \eqref{aprhproperty}. For $\Gamma> \Gamma_{0}$, since $\Phi_{y}$ is continuous and
		$\lim\limits_{\rho \rightarrow 0}\Phi_{y}(\rho) > \Gamma^{p-\delta_{0}} $,
		there exists an exit time radius $\rho_{y} \in (0,  \frac{
			r_{2}-r_{1} }{10} )$ such that
		\begin{align}\label{exittime}
			\Phi_{y}(\rho_{y}) = \Gamma^{p-\delta_{0}} \quad \textrm{and} \quad \Phi_{y}(\rho)<\Gamma^{p-\delta_{0}} \quad \textrm{if} \enspace  \rho \in (\rho_{y},  r_{2}-r_{1} ].
		\end{align}
		We now consider the family of balls \(\{B_{\rho_{y}}(y): y\in E^{r_{1}}_{\Gamma}\}\),
	 which covers \(E^{r_{1}}_{\Gamma}\). Applying Vitali's covering
		lemma, we can select a countable collection of pairwise disjoint balls
		\(\{B_{\rho_i}(y_i): y_i\in E^{r_{1}}_{\Gamma}\}\) such that
		\begin{align}
		 \notag	E^{r_{1} }_{\Gamma} \subset \bigcup_{i\geq 1} B_{5\rho_{i}}(y_{i}) \cup \textrm{negligible set},
		\end{align}
		where we have denoted $\rho_{i}=\rho_{y_{i}}$. We write $B^{i}_{r}:= B_{5\rho_{i}}(y_{i})$ and so $B^{i}_{2r} = B_{10\rho_{i}}(y_{i})$. It follows from \eqref{exittime} that
		\begin{align}
		\notag \frac{1}{\omega(B^{i}_{2r})}\int_{B^{i}_{2r}}  |Du|^{p-\delta_{0}} \omega \, dx \leq \Gamma^{p-\delta_{0}},  \quad  \frac{1}{\omega(B^{i}_{2r})} \int_{B^{i}_{2r}}   |\mathbf{f}|^{p-\delta_{0}} \omega \, dx \leq (\delta_{0}\Gamma)^{p-\delta_{0}},
		\end{align}
		which shows that the required conditions in \eqref{intlambda} are satisfied. According to Lemma \ref{compari}, we find that for $\Gamma> \Gamma_{0}$ and for any $T \geq 1$,
			\begin{align}\label{loclevelsetesti1}
				\nonumber &  \int_{\{x \in B^{i}_{r} : |Du|>T \Gamma \}}  |Du|^{p-\delta_{0}} \omega \, dx  \leq  c \int_{\{x \in B^{i}_{r} : |Du|>T \Gamma \}}  |V(Du)|^{2(1-\frac{\delta_{0}}{p})} \omega \, dx  \\
				\nonumber &  \hspace{5mm}  \leq   c \int_{\{x \in B^{i}_{r} : |Du|>T\Gamma \}}  |V(Du)-V(Dw)|^{2(1-\frac{\delta_{0}}{p})} \omega  \, dx   \\
				\nonumber &  \hspace{10mm}  + c \int_{\{x \in B^{i}_{r} : |Du|>T\Gamma \}}  |V(Dw)|^{2(1-\frac{\delta_{0}}{p})} \omega \, dx \\
				&  \hspace{5mm}   \leq  \varepsilon  \Gamma^{p-\delta_{0}} \omega(B^{i}_{2r}) + c_{0} \int_{\{x \in B^{i}_{r} : |Du|>T\Gamma \}}  |Dw|^{p-\delta_{0}} \omega \, dx
			\end{align}
		 for some constant $c_{0}>0$ depending only on $n$ and $p$. Then the last term in the right-hand side of \eqref{loclevelsetesti1} can be estimated as
			\begin{align} 
				\notag & \int_{\{x \in B^{i}_{r} : |Du|>T\Gamma \}}  |Dw|^{p-\delta_{0}} \omega \, dx  \\
				\notag &  \leq  \omega(\{x \in B^{i}_{r} : |Du|>T\Gamma \})^{\frac{\delta_{0}+\sigma}{p+\sigma}}  \left( \int_{B^{i}_{r}} |Dw|^{p+\sigma} \omega \, dx  \right)^{\frac{p-\delta_0}{p+\sigma}} \\
				\notag &   \leq  c  \left( \int_{\{x \in B^{i}_{r} : |Du|>T \Gamma \}}  \left(\frac{|Du|}{T\Gamma}\right)^{p-\delta_{0}}  \omega \, dx \right)^{\frac{\delta_{0}+\sigma}{p+\sigma}}   \Gamma^{p-\delta_{0}}  \omega(B^{i}_{2r})^{\frac{p-\delta_{0}}{p+\sigma}} \\
				\notag &   =  c  \left( \int_{\{x \in B^{i}_{r} : |Du|>T \Gamma \}}  |Du|^{p-\delta_{0}}  \omega  \, dx \right)^{\frac{\delta_{0}+\sigma}{p+\sigma}}  \left[  T^{-\delta_{0}-\sigma} \Gamma^{p-\delta_{0}}  \omega(B^{i}_{2r}) \right]^{\frac{p-\delta_{0}}{p+\sigma}} \\
				\notag &     \leq   \frac{1}{2c_{0}} \int_{\{x \in B^{i}_{r} : |Du|>T \Gamma \}}   |Du|^{p-\delta_{0}} \omega \, dx +  c  T^{-\delta_{0}-\sigma} \Gamma^{p-\delta_{0}}  \omega(B^{i}_{2r}),
			\end{align}
			where the second inequality follows from the weighted Chebyshev inequality, while the last inequality is obtained by applying Young's inequality. Then we find
		\begin{align}\label{loclevelsetesti2}
			\nonumber &     \int_{\{x \in B^{i}_{r} : |Du|>T \Gamma \}}  |Du|^{p-\delta_{0}}  \omega \, dx \leq ( 2  \varepsilon + c^{*}  T^{-\delta_{0}-\sigma}) \Gamma^{p-\delta_{0}}  \omega(B^{i}_{2r})  \\
			& \hspace{15mm}\leq  c( \varepsilon + c^{*}  T^{-\delta_{0}-\sigma} ) \Gamma^{p-\delta_{0}}   \omega(B_{\rho_{i}}(y_{i}))   
		\end{align}
		for some $c^{*}$ depending only on $n, p, \nu, L, \Lambda$ and $ [\omega]_{A_{p}}$. Now we estimate
		\begin{align}\label{loclevelsetesti3}
			\nonumber &  \omega(B_{\rho_{i}}(y_{i})) =  \frac{1}{\Gamma^{p-\delta_{0}}}  \int_{B_{\rho_{i}}(y_{i}) } \left\{ |Du|^{p-\delta_{0}}
			+ 	\left(\frac{|\mathbf{f}|}{\delta_{0}}\right)^{p-\delta_{0}} \right\} \omega \, dx \\
			\nonumber &   \hspace{5mm}  \leq    \frac{1}{\Gamma^{p-\delta_{0}}}  \int_{\{x \in B_{\rho_{i}}(y_{i}) : |Du|> \frac{\Gamma}{4} \} }  |Du|^{p-\delta_{0}} \omega \, dx \\
			&\hspace{10mm}  + \frac{1}{\Gamma^{p-\delta_{0}}}  \int_{\{x \in B_{\rho_{i}}(y_{i}) : |\mathbf{f}|> \frac{\delta_{0}\Gamma}{4} \} }  	\left(\frac{|\mathbf{f}|}{\delta_{0}}\right)^{p-\delta_{0}} \omega  \, dx + \frac{  \omega(B_{\rho_{i}}(y_{i})) }{2}.
		\end{align}
		It follows from \eqref{loclevelsetesti2} and \eqref{loclevelsetesti3} that
		\begin{align} 
			\notag &     \int_{\{x \in B^{i}_{r} : |Du|>T \Gamma \}}  |Du|^{p-\delta_{0}}  \omega \, dx \\
			\notag & \hspace{5mm}   \leq  c ( \varepsilon + c^{*} T^{-\delta_{0}-\sigma} )  \left[  \int_{\{x \in B_{\rho_{i}}(y_{i}) : |Du|> \frac{\Gamma}{4} \} }  |Du|^{p-\delta_{0}} \omega \, dx \right. \\
			\notag &   \hspace{40mm} \left. +   \int_{\{x \in B_{\rho_{i}}(y_{i}) :|\mathbf{f}|> \frac{\delta_{0}\Gamma}{4} \} }  \left(\frac{|\mathbf{f}|}{\delta_{0}}\right)^{p-\delta_{0}} \omega   \, dx \right] .
		\end{align}
	Let \(\mathcal{E}^{s}_{\Gamma}\) denote the upper level set of
	\(|\mathbf{f}|\) given by
	\[
	\mathcal{E}^{s}_{\Gamma}=\{x\in B_s(z): |\mathbf{f}|>\Gamma\}.
	\]
	Recall that $ E^{r_1}_{T\Gamma} $ is covered by 
	the family \(\{B^i_r\}_{i=1}^{\infty}\), where
	\(\{B_{\rho_i}(y_i)\}_{i=1}^{\infty}\) is a collection of mutually disjoint
	balls. Summing the above inequality over all balls in the covering
	$\{B_r^i\}_{i=1}^\infty $, we obtain
		\begin{align} 
			\notag &    \int_{E^{r_{1}}_{T\Gamma}}  |Du|^{p-\delta_{0}} \omega \, dx  \\
			\notag &     \leq  \sum\limits_{i \geq 1}  \int_{\{x \in B^{i}_{r} : |Du|>T \Gamma \}}  |Du|^{p-\delta_{0}}  \omega \, dx  \\
			\notag &      \leq      c ( \varepsilon + c^{*} T^{-\delta_{0}-\sigma} )   \left[  \sum\limits_{i \geq 1}  \int_{\{x \in B_{\rho_{i}}(y_{i}) : |Du|> \frac{\Gamma}{4} \} }  |Du|^{p-\delta_{0}}  \omega  \, dx \right. \\
			\notag &  \hspace{40mm}  \left. +  \sum\limits_{i \geq 1}   \int_{\{x \in B_{\rho_{i}}(y_{i}) : |\mathbf{f}|> \frac{\delta_{0}\Gamma}{4} \} } \left(\frac{|\mathbf{f}|}{\delta_{0}}\right)^{p-\delta_{0}} \omega \, dx \right] \\
			\notag &      \leq   c ( \varepsilon + c^{*} T^{-\delta_{0}-\sigma} )    \left[   \int_{ \bigcup_{i\geq 1} \{x \in B_{\rho_{i}}(y_{i}) : |Du|> \frac{\Gamma}{4} \} }  |Du|^{p-\delta_{0}} \omega \, dx \right. \\
			\notag & \hspace{40mm}  \left. +     \int_{ \bigcup_{i\geq 1} \{x \in B_{\rho_{i}}(y_{i}) : |\mathbf{f}|> \frac{\delta_{0}\Gamma}{4} \} } \left(\frac{|\mathbf{f}|}{\delta_{0}}\right)^{p-\delta_{0}}  \omega  \, dx \right] \\
		 \notag&      \leq   c ( \varepsilon + c^{*} T^{-\delta_{0}-\sigma} )    \left[ \int_{E^{r_{2}}_{\Gamma/4}} |Du|^{p-\delta_{0}} \omega \, dx +   \int_{\mathcal{E}^{r_{2}}_{\delta_{0}\Gamma/4} }   \left(\frac{|\mathbf{f}|}{\delta_{0}}\right)^{p-\delta_{0}} \omega  \, dx \right]
		\end{align}
 for $\Gamma > \Gamma_{0}$. Thus, performing a change of variables in \(\Gamma\), we obtain
		\begin{align} 
			\notag  &    \int_{E^{r_{1}}_{\Gamma}} |Du|^{p-\delta_{0}} \omega \, dx     \\
			\notag&   \leq  c ( \varepsilon + c^{*} T^{-\delta_{0}-\sigma} )  \left[ \int_{E^{r_{2}}_{\Gamma/4T}} |Du|^{p-\delta_{0}} \omega \, dx +   \int_{\mathcal{E}^{r_{2}}_{\delta_{0}\Gamma/4T} }   \left(\frac{|\mathbf{f}|}{\delta_{0}}\right)^{p-\delta_{0}} \omega \, dx \right]
		\end{align}
		for $\Gamma> T\Gamma_{0}$.
		
		We now turn to the $L^\gamma$ estimates. Since the case $\gamma = p-\delta_{0}$ follows directly from Inclusion $B_{R} \Subset B_{2R}$, we henceforth consider the case where \(\gamma \in (p-\delta_0,p+\delta_0]\). For a measurable function \(g\), let
		\[
		[g]_t:=\min\{|g|,t\}
		\]
		denote its truncation at level \(t\). For \(t>t_0:=T\Gamma_0\), an application of Fubini's theorem yields
		\begin{align}\label{fubiniinteg1}
			\nonumber &    \int_{B_{r_{1}}(z)} |Du|^{p-\delta_{0}}  \left[|Du|\right]_{t}^{\gamma-p+\delta_{0}} \omega \, dx  \\
			\nonumber &= \left( \gamma-p+\delta_{0} \right)  \int_0^{t} \Gamma^{\gamma-p+\delta_{0}-1} \int_{E^{r_{1}}_{\Gamma}} |Du|^{p-\delta_{0}} \omega \, dx \, d\Gamma \\
			\nonumber & \leq  \left( \gamma-p+\delta_{0} \right) \int_0^{t_{0}} \Gamma^{\gamma-p+\delta_{0}-1} \int_{E^{r_{1}}_{\Gamma}} |Du|^{p-\delta_{0}} \omega \, dx \, d\Gamma\\
			\nonumber &  \hspace{5mm}+  \left( \gamma-p+\delta_{0} \right) \int_{t_{0}}^{t}  \Gamma^{\gamma-p+\delta_{0}-1} \int_{E^{r_{1}}_{\Gamma}} |Du|^{p-\delta_{0}} \omega \, dx \, d\Gamma \\
			\nonumber & \leq (T\Gamma_{0})^{\gamma-p+\delta_{0}}  \int_{B_{r_{2}}(z)} |Du|^{p-\delta_{0}} \omega \, dx  \\
			\nonumber &  \hspace{5mm}+  \left( \gamma-p+\delta_{0} \right) \int_{t_{0}}^{t}  \Gamma^{\gamma-p+\delta_{0}-1} \int_{E^{r_{1}}_{\Gamma}} |Du|^{p-\delta_{0}} \omega \, dx \, d\Gamma \\
			\nonumber &  \hspace{-2.5mm}  \underset{\eqref{lambdazero}}{\leq}  T^{\gamma-p+\delta_{0}} \Gamma_{0}^{\gamma}  \omega(B_{r_{2}}(z)) \\
			\nonumber &  \hspace{5mm} + c ( \varepsilon + c^{*} T^{-\delta_{0}-\sigma} ) \left( \gamma-p+\delta_{0} \right) \left[  \int_{t_0}^{t}  \Gamma^{\gamma-p+\delta_{0}-1}  \int_{E^{r_{2}}_{\Gamma/4T}} |Du|^{p-\delta_{0}} \omega \, dx  \, d\Gamma \right.\\
			&   \hspace{25mm} \left. +  \int_{t_0}^{t}  \Gamma^{\gamma-p+\delta_{0}-1} \int_{\mathcal{E}^{r_{2}}_{\delta_{0}\Gamma/4T} }   \left(\frac{|\mathbf{f}|}{\delta_{0}}\right)^{p-\delta_{0}} \omega \, dx  \, d\Gamma \right] .
		\end{align}
		To estimate the last two integrals, we again use Fubini's theorem and a change
		of variables to obtain
		\begin{align}\label{fubiniinteg2}
			\nonumber &      \int_{t_0}^{t}  \Gamma^{\gamma-p+\delta_{0}-1}  \int_{E^{r_{2}}_{\Gamma/4T}} |Du|^{p-\delta_{0}} \omega \, dx  \, d\Gamma \\
			\nonumber &    \leq \frac{(4T)^{\gamma-p+\delta_{0}}}{\gamma-p+\delta_{0}}\int_{B_{r_{2}}(z)} |Du|^{p-\delta_{0}}  \left[|Du|\right]_{t/4T}^{\gamma-p+\delta_{0}} \omega \, dx \\
			&    \leq \frac{(4T)^{\gamma-p+\delta_{0}}}{\gamma-p+\delta_{0}} \int_{B_{r_{2}}(z)} |Du|^{p-\delta_{0}}  \left[|Du|\right]_{t}^{\gamma-p+\delta_{0}} \omega \, dx .
		\end{align}
		Likewise, we have
		\begin{align}\label{fubiniinteg3}
		     \notag & \int_{t_0}^{t}  \Gamma^{\gamma-p+\delta_{0}-1} \int_{\mathcal{E}^{r_{2}}_{\delta_{0}\Gamma/4T} }   \left(\frac{|\mathbf{f}|}{\delta_{0}}\right)^{p-\delta_{0}} \omega \, dx  \, d\Gamma   \\
		     &  \leq \frac{(4T)^{\gamma-p+\delta_{0}}}{\delta_{0}^{\gamma}(\gamma-p+\delta_{0})} \int_{B_{r_{2}}(z)}|\mathbf{f}|^{\gamma} \omega \, dx .
		\end{align}
		Combining \eqref{fubiniinteg1}, \eqref{fubiniinteg2} and \eqref{fubiniinteg3}, it follows that
		\begin{align}
			\notag &    \int_{B_{r_{1}}(z)} |Du|^{p-\delta_{0}}  \left[|Du|\right]_{t}^{\gamma-p+\delta_{0}} \omega \, dx \\
			\notag &    \leq T^{\gamma-p+\delta_{0}} \Gamma_{0}^{\gamma}  \omega(B_{r_{2}}(z))  \\
			 \nonumber &  \hspace{5mm} + c_{*}( \varepsilon T^{\gamma-p+\delta_{0}} + c^{*}T^{\gamma-p-\sigma} )  \int_{B_{r_{2}}(z)} |Du|^{p-\delta_{0}}  \left[|Du|\right]_{t}^{\gamma-p+\delta_{0}} \omega \, dx   \\
			\notag &  \hspace{5mm} + \frac{ c_{*}( \varepsilon T^{\gamma-p+\delta_{0}} + c^{*}T^{\gamma-p-\sigma} )  }{\delta_{0}^{\gamma} } \int_{B_{r_{2}}(z)} |\mathbf{f}|^{\gamma} \omega  \, dx  \\
			\notag &  \leq T^{\gamma-p+\delta_{0}} \Gamma_{0}^{\gamma}  \omega(B_{r_{2}}(z))  \\
			\notag &  \hspace{5mm}  +  c_{*}( \varepsilon T^{2} + c^{*}T^{-\frac{\sigma}{2}} )  \int_{B_{r_{2}}(z)} |Du|^{p-\delta_{0}}  \left[|Du|\right]_{t}^{\gamma-p+\delta_{0}} \omega \, dx   \\
			\notag &  \hspace{5mm} +  \frac{ c_{*}( \varepsilon T^{2} + c^{*}T^{-\frac{\sigma}{2}} )  }{\delta_{0}^{p} } \int_{B_{r_{2}}(z)}  |\mathbf{f}|^{\gamma} \omega  \, dx
		\end{align}
		for some $c_{*}$ depending only on $n, p, \nu, L, \Lambda$ and $ [\omega]_{A_{p}}$, where we have used the fact that $\gamma \leq p+\delta_{0} \leq p+\frac{\sigma}{2}$. By choosing \(T\) sufficiently large and \(\varepsilon\) sufficiently small, , depending also on the doubling constant in \eqref{double}, we deduce
		\begin{align} 
			\nonumber &   \frac{1}{\omega(B_{r_{1}}(z))}\int_{B_{r_{1}}(z)} |Du|^{p-\delta_{0}}  \left[|Du|\right]_{t}^{\gamma-p+\delta_{0}} \omega  \, dx  \\
			\nonumber &  \  \leq \left(  \frac{ c R^{np}}{(r_{2}-r_{1})^{np}\omega(B_{R}(z))} \int_{B_{R}(z)}  \left\{ 
			|Du|^{p-\delta_{0}}
			+ 	\left(\frac{|\mathbf{f}|}{\delta_{0}}\right)^{p-\delta_{0}} \right\} \omega \, dx \right)^{\frac{\gamma}{p-\delta_{0}}}   \\
		\nonumber &    +   \frac{1}{2\omega(B_{r_{2}}(z))}\int_{B_{r_{2}}(z)}   |Du|^{p-\delta_{0}}  \left[|Du|\right]_{t}^{\gamma-p+\delta_{0}} \omega   \, dx  + \frac{c}{\omega(B_{R}(z))}\int_{B_{R}(z)}  |\mathbf{f}|^{\gamma} \omega  \, dx ,
		\end{align}
		where we have used the definition of $\Gamma_{0}$ in \eqref{lambdazero} and \eqref{aprhproperty}. By applying Lemma \ref{algeb} to the function
		$$
		\phi(s) := \frac{1}{\omega(B_{s})}\int_{B_{s}(z)}     |Du|^{p-\delta_{0}}  \left[|Du|\right]_{t}^{\gamma-p+\delta_{0}} \omega  \, dx
		$$
		with $\beta = \frac{np\gamma}{p-\delta_{0}} \in [n\gamma,np\gamma]$, we deduce the following estimate:
		\begin{align} 
			\notag &   \frac{1}{\omega(B_{\frac{R}{2}}(z))} \int_{B_{\frac{R}{2}}(z)}    |Du|^{p-\delta_{0}}  \left[|Du|\right]_{t}^{\gamma-p+\delta_{0}} \omega  \, dx  \\
			\notag &   \leq  \left(     \frac{c}{\omega(B_{R}(z))} \int_{B_{R}(z)}   |Du|^{p-\delta_{0}} \omega 
			\, dx \right)^{\frac{\gamma}{p-\delta_{0}}}   +    \frac{c}{\omega(B_{R}(z))}  \int_{B_{R}(z)}|\mathbf{f}|^{\gamma} \omega  \, dx  .
		\end{align}
		Letting $t \rightarrow \infty$, we obtain the desired estimate
		\begin{align} 
			\notag &   \frac{1}{\omega(B_{\frac{R}{2}}(z))} \int_{B_{\frac{R}{2}}(z)}    |Du|^{\gamma}   \omega  \, dx  \\
			\notag &   \leq  \left(     \frac{c}{\omega(B_{R}(z))} \int_{B_{R}(z)}   |Du|^{p-\delta_{0}} \omega 
			\, dx \right)^{\frac{\gamma}{p-\delta_{0}}}   +    \frac{c}{\omega(B_{R}(z))}  \int_{B_{R}(z)}|\mathbf{f}|^{\gamma} \omega  \, dx  .
		\end{align}
		Using the doubling property of \(\omega\), multiplying both sides by
		\(\omega(B_{R/2}(z))\), and performing a rescaling, we can rewrite the
		estimate in the form of Theorem  \ref{main1}.
		$$
		\int_{B_{R}} | D u |^{\gamma} \, \omega\, dx
		\le C \left[ \omega(B_{2R})^{\frac{p-\delta_0-\gamma}{p-\delta_0}}
		\left( \int_{B_{2R}} | D u |^{p-\delta_0} \, \omega\, dx \right)^{\frac{\gamma}{p-\delta_0}}
		+ \int_{B_{2R}} | \mathbf{f} |^{\gamma} \, \omega\, dx
		\right]
		$$
		for some constant $C$ depending only on $n,p, \nu,L,\Lambda$ and $[\omega]_{A_{p}}$.
	\end{proof}

\providecommand{\bysame}{\leavevmode\hbox to3em{\hrulefill}\thinspace}
\providecommand{\MR}{\relax\ifhmode\unskip\space\fi MR }
\providecommand{\MRhref}[2]{%
	\href{http://www.ams.org/mathscinet-getitem?mr=#1}{#2}
}
\providecommand{\href}[2]{#2}

\end{document}